\documentclass{article}

\usepackage{color}
 \catcode`@=11
\usepackage{color}
\usepackage{amsthm,amssymb,amsfonts,mathrsfs}
\usepackage{amsmath}
\catcode`@=11 \@addtoreset{equation}{section}

\catcode`@=12

\newtheorem{theorem}{Theorem}[section]
\newtheorem{proposition}{Proposition}[section]
\newtheorem{lemma}{Lemma}[section]

\theoremstyle{definition}

\newtheorem{remark}{Remark}

\def\supp{\mathop{\rm supp}\nolimits}
\usepackage[colorlinks=true,urlcolor=red,linkcolor=blue,citecolor=red,bookmarks=red]{hyperref}
\title{Controllability of degenerate parabolic equation with memory}
\author{{\sc Brahim Allal}\thanks{The author thanks the INdAM - GNAMPA Project 2019 {\it Controllabilit\`a di PDE in modelli fisici e in scienze della vita}.}\\
Facult\'e des Sciences et Techniques\\
Universit\'e Hassan 1er\\
Laboratoire MISI, B.P. 577\\
Settat 26000, Morocco
\\ email: b.allal@uhp.ac.ma
\\
{\sc Genni Fragnelli}\thanks{The author is a member of the Gruppo Nazionale per l'Analisi Ma\-te\-matica, la Probabilit\`a e le loro Applicazioni (GNAMPA) of the
Istituto Nazionale di Alta Matematica (INdAM) and she is supported by the FFABR {\it Fondo
per il finanziamento delle attivit\`a base di ricerca} 2017, by the INdAM - GNAMPA Project 2019 {\it Controllabilit\`a di PDE in modelli fisici e in scienze della vita},  by  Fondi di Ateneo 2015/16 of the University of  Bari {\em Problemi differenziali non linearii} and by PRIN 2017-2019 {\it Qualitative and quantitative aspects of nonlinear PDEs}.}\\
Dipartimento di Matematica\\ Universit\`{a} di Bari "Aldo Moro"\\
Via
E. Orabona 4\\ 70125 Bari - Italy\\ email: genni.fragnelli@uniba.it\\}

\date{}
\begin{document}
\maketitle

\begin{abstract}
In this paper, we analyze the null controllability property for a degenerate parabolic equation involving memory terms with a locally distributed control. We first derive a null controllability result for a nonhomogeneous degenerate heat equation via new Carleman estimates with weighted time functions that do not blow up at $t= 0$. Then this result is successfully used with a classical fixed point to obtain null controllability for the initial memory system.
\end{abstract}
\section{Introduction}\label{Sect_Introduction}
In this work we are concerned with the null controllability result for a degenerate parabolic equation with memory by a distributed control force. More precisely, we consider the following controlled system:
\begin{equation}\label{sys-memory1}
\left\{
\begin{array}{lll}
\displaystyle y_t - (a(x)  y_x )_x = \int\limits_0^t b(t,s,x)  y(s,x) \, ds + 1_{\omega} u &  & (t, x) \in Q=(0, T) \times (0, 1), \\
y(t,1)= 0, & & t \in (0, T),  \\
\begin{cases}
& y(t, 0) = 0,   \qquad   \quad (WD), \\
& (a y_x ) (t, 0)= 0 , \quad (SD), \\
\end{cases}  & & t \in (0, T),\\
y(0,x)= y_{0}(x),  & & x \in  (0, 1).
\end{array}
\right.
\end{equation}
Here, $\omega \Subset (0, 1)$ is a non-empty open set, $1_{\omega}$ is the corresponding characteristic function, $u = u(t, x)$ is the control function, $y = y(t, x)$ is the
state and $b=b(t,s,x)\in L^\infty((0, T) \times Q)$ is a memory kernel. Moreover, the diffusion coefficient $a$ 
vanishes at the boundary $x=0$ (i.e., $a(0) = 0$) and can be either weakly degenerate (WD), i.e.,
\begin{equation}\label{hyp_WD_left}
\left\{
\begin{array}{lll}
a \in C([0,1]) \cap C^1((0,1]), \; a(0)=0, \; a>0 \quad \text{in}\quad (0, 1], \\
\exists \, \alpha \in [0,1),\quad \text{such that}\quad x a'(x) \leq \alpha a(x),\quad \forall \, x \in [0,1],
\end{array}
\right.
\end{equation}
or strongly degenerate (SD), i.e.,
\begin{equation}\label{hyp_SD_left}
\left\{
\begin{array}{lll}
a \in C^1([0,1]), \; a(0)=0, \; a>0 \quad \text{in}\quad (0, 1], \\
\exists \, \alpha \in [1,2),\quad \text{such that}\quad x a'(x) \leq \alpha a(x),\quad \forall \, x \in [0,1],\\
\left\{
\begin{array}{ll}
\exists \, \beta \in (1, \alpha],\, x \mapsto \dfrac{a(x)}{x^{\beta}}\quad \text{is nondecreasing near}\quad 0,\quad \text{if}\quad \alpha > 1, \\
\exists \, \beta \in (0, 1),\, x \mapsto \dfrac{a(x)}{x^{\beta}}\quad \text{is nondecreasing near}\quad 0,\quad \text{if}\quad \alpha=1.
\end{array}
\right.
\end{array}
\right.
\end{equation}
A typical example of coefficient $a$ is the following:
$$ a(x) = x^{\alpha}, \quad \alpha \in (0, 2). $$

The null controllability of parabolic equations without memory (i.e. $b\equiv0$) is by now well understood, for both uniformly and degenerate diffusion coefficient, by means of distributed and boundary controls (see \cite{Hajjaj2013, Alabau2006, bfm2018, CMV2016, fm2013, fm2016, FI1996} and the references therein).

On the other hand, in the presence of memory terms, much less is known on the controllability of the underlying system. 

When $a=b=1$, S. Guerrero and O. Imanuvilov prove in \cite{Guerrero2013} that \eqref{sys-memory1} fails to be null controllable with a boundary control. Indeed, there exists a set of initial states that cannot be driven to $0$ in any positive final time. Then, similar result is proved by X. Zhou and H. Gao in \cite{Zhou2014} whenever $b$ is a non-trivial constant; in this paper it is also proved that the approximate controllability holds. Later on, these results are extended  in \cite{Zhou2018} to the context of one dimensional degenerate parabolic equation. In particular, the authors assume that $a(x) = x^{\alpha}$, being $x\in (0, 1)$, $ 0\leq \alpha < 1 $ and prove that the null controllability of \eqref{sys-memory1} fails whereas the approximate property holds in a suitable state space with a boundary control acting at the extremity $x = 0$ or $x = 1$.

Thus, it is  important to see which kind of conditions on $b$ we have to require so that the null controllability of \eqref{sys-memory1} holds. In \cite{Lavanya2009, Saktiv2008} R. Lavanya, K. Balachandran and B.R. Nagaraj obtained the null controllability of a nonlinear  and non degenerate version of \eqref{sys-memory1} assuming that the memory kernel  is sufficiently smooth and vanishes at the neighborhood of initial and final times. In particular,
\begin{equation}\label{condit_lavanya}
b(t,s,x)\equiv b(t,s) \quad \text{and} \quad \supp b( \cdot, s) \Subset (t_0, t_1), \quad 0 < t_0 < t < t_1 < T, \quad \forall s\in(0, T).
\end{equation}
The proof relies on Carleman estimates and a fixed point method. This assumption has been relaxed by Q. Tao and H. Gao in  \cite{TaoGao16}, where the authors showed that null controllability holds provided $b$ fulfills
\begin{equation}\label{hyp_kernel_Tao_Gao}
e^{\frac{C}{(T-t)}} b \in  L^{\infty}((0, T) \times Q)    
\end{equation}
for some positive constant $C$.

For related results on this subject, we refer to\cite{Zuazua2017} for wave equation, \cite{Barbu2000} for viscoelasticity equation, \cite{Munoz2003} for thermoelastic system and \cite{Yong2005} in the case of heat equation with hyperbolic memory kernel (see also the bibliography therein).

The purpose of this paper is to give a suitable condition on the memory kernel  $b$ in such a way that the degenerate parabolic equation with memory \eqref{sys-memory1} is null controllable, that is there exists a control $u \in L^2(Q)$ such that the associated solution of \eqref{sys-memory1}, corresponding to the initial data $y_0 \in L^2(0, 1)$,
satisfies
\begin{equation*}
y(T, \cdot) = 0 \qquad \text{in} \; (0, 1).
\end{equation*}
We include here a brief description of the proof strategy:
in a first step, we focus on the following nonhomogeneous degenerate parabolic system
\begin{equation}\label{sys-nonhom}
\left\{
\begin{array}{lll}
\displaystyle y_t - (a(x)  y_x )_x = f + 1_{\omega} u &  & (t, x) \in Q, \\
y(t,1)= 0 , & & t \in (0, T),  \\
\begin{cases}
& y(t, 0) = 0,   \qquad   \quad (WD), \\
& (a y_x ) (t, 0)= 0 , \quad (SD), \\
\end{cases}  & & t \in (0, T),\\
y(0,x)= y_{0}(x),  & & x  \in  (0, 1),
\end{array}
\right.
\end{equation}
for a given function $f \in L^2(Q)$.

In particular, we establish suitable Carleman estimates for the associated adjoint problem using some classical weight time functions that blow up to $+\infty$ as $t\rightarrow 0^-, T^+$. Then, using a weight time function not exploding in the neighborhood of $t=0$, we derive a new  modified Carleman estimate that would allow us to show null controllability of the underlying parabolic equation. As a consequence, we deduce null controllability result for some problems similar to the degenerate parabolic equation with memory. Finally, this controllability result combined with an appropriate application of Kakutani's fixed point Theorem allows us to obtain the null controllability result for the original system \eqref{sys-memory1} under a suitable condition on the kernel $b$. 

\begin{remark}
We believe that the null controllability of system \eqref{sys-memory1} can be obtained also following the same ideas in \cite{Saktiv2008, Lavanya2009}. More precisely, by means of classical duality arguments, the null controllability property can be reduced to an observability inequality for the adjoint parabolic problem
\begin{equation}\label{sys-adj-memory}
\left\{
\begin{array}{lll}
\displaystyle - v_t - (a(x)  v_x )_x = \int\limits_t^T b(s,t,x)  v(s,x) \, ds &  & (t, x) \in Q, \\
v(t,1)= 0 , & & t \in (0, T),  \\
\begin{cases}
& v(t, 0) = 0,   \qquad   \quad (WD), \\
& (a v_x )  (t, 0)= 0 , \quad (SD), \\
\end{cases}  & & t \in (0, T),\\
v(T,x)= v_{T}(x),  & & x \in  (0, 1),
\end{array}
\right.
\end{equation}
where $v_T \in L^2(Q)$ and $g\in L^2(Q)$. 

Such an inequality is proved by R. Lavanya and K. Balachandran in the aforementioned reference through the use of a new Carleman estimate for \eqref{sys-adj-memory} under a strict restriction on the memory kernel. Indeed, in order to treat the integral term in \eqref{sys-adj-memory}, the coefficient $b$ need to be sufficiently smooth and to satisfy condition \eqref{condit_lavanya}. One could expects the same condition for system \eqref{sys-memory1}.

However, in this paper, we follow the methodology used in  \cite{TaoGao16} for the treatment of nondegenerate equation which permits us to show that system 
\eqref{sys-memory1} is null controllable provided the coefficient $b$ satisfies only some exponential decay at the final time $t=T$ (see \eqref{hypoth_kernel}). 
\end{remark}
The outline of this paper is as follows: Section \ref{sect_well_posed} is devoted to the well-posedness of systems \eqref{sys-memory1} and \eqref{sys-nonhom} in suitable weighted spaces. In Section \ref{sect_carleman_estimate}, we develop a new Carleman estimate for the adjoint problem to the nonhomogeneous parabolic equation \eqref{sys-nonhom} and, in Section \ref{sect_null_control_nonhomog}, we apply such an estimate to deduce null controllability for \eqref{sys-nonhom}. In Section \ref{sect_null_control_memory}, using the Kakutani's fixed point Theorem, we prove the null controllability result for the degenerate parabolic equation with memory \eqref{sys-memory1} under suitable condition on the memory kernel. Finally, in Section \ref{sect_comments}, we discuss various extensions of our result and give some perspectives related to this work.


\section{Well-posedness results} \label{sect_well_posed} 
The goal of this section is to study the well-posedness results for \eqref{sys-memory1} and \eqref{sys-nonhom}. First, we recall the following weighted Sobolev spaces (in the sequel, a.c. means  absolutely continuous):

In the (WD) case:
\begin{align*}
H_a^1:= \Big\{ y \in L^2(0,1): y\; &\text{a.c. in}\, [0,1],\; \sqrt{a}y_x \in L^2(0,1)\;\text{and}\; y(1)=y(0)=0 \Big\}
\end{align*}
and
\begin{align*}
H_a^2:= \Big\{ y \in H_a^1(0, 1): ay_x \in H^1(0,1)\Big\}.
\end{align*}
In the (SD) case:
\begin{align*}
H_a^1:= \Big\{ y \in L^2(0,1): y\, &\text{locally a.c. in}\, (0,1],\quad\sqrt{a}y_x \in L^2(0,1)\,\text{and}\, y(1)=0 \Big\}
\end{align*}
and
\begin{align*}
H_a^2:&= \Big\{ y \in H_a^1(0, 1): ay_x \in H^1(0,1)\Big\}\\
&=\Big\{ y \in L^2(0,1): y\, \text{locally a.c. in}\, (0,1], ay\in H^1_0(0,1),\\
&\qquad ay_x \in H^1(0,1) \,\text{and}\, (ay_x)(0)=0 \Big\}.
\end{align*}
In both cases, the norms are defined as follow
\begin{align*}
\| y \|_{H_a^1}^2:=  \| y \|_{L^2(0,1)}^2 +  \| \sqrt{a} y_x \|_{L^2(0,1)}^2,\qquad
\| y \|_{H_a^2}^2 :=  \| y \|_{H_a^1}^2 +  \| (a y_x )_x \|_{L^2(0,1)}^2.
\end{align*}

We recall the following well-posedness result for system \eqref{sys-nonhom} (see, for instance, \cite{Alabau2006, Campiti1998}).
\begin{proposition}\label{prop-Well-posed_nonhom}
Assume that $y_0 \in L^2(0,1)$, $f \in L^2(Q)$ and $u\in L^2(Q)$. Then, system \eqref{sys-nonhom} admits a unique solution
\begin{equation}\label{Class_W_T_memory}
y \in W_T := L^2(0, T; H_a^1(0,1))\cap C([0, T]; L^2(0,1))
\end{equation}
such that
\begin{equation}\label{energy-nonhom1}
\|y\|_{L^2(0, T; H_a^1(0,1))} +\|y\|_{C([0, T]; L^2(0,1))} 
\leq C \Big(\|y_0\|_{L^2(0,1)} + \|f\|_{L^2(Q)} + \|1_{\omega} u\|_{L^2(Q)} \Big),
\end{equation}
for some positive constant $C$.
Moreover, if  $y_0\in H_a^{1}(0,1)$, then
\begin{equation*}
y \in Z_T := L^2(0, T; H_{a}^{2}(0,1))\cap H^1(0, T; L^2(0,1))
\end{equation*}
and
\begin{equation}\label{energy-nonhom2}
\|y\|_{L^2(0, T; H_a^{2}(0,1))} +\|y\|_{H^1(0, T; L^2(0,1))} 
\leq C \Big(\|y_0\|_{H_a^{1}(0,1)} + \|f\|_{L^2(Q)} + \|1_{\omega} u\|_{L^2(Q)} \Big),
\end{equation}
for some positive constant $C$.
\end{proposition}

Existence and uniqueness of solution for system \eqref{sys-memory1} are established in the following result:

\begin{proposition}\label{prop-Well-posed_memory}
Assume that $y_0 \in L^2(0,1)$ and $u\in L^2(Q)$. Then, system \eqref{sys-memory1} admits a unique solution
$
y \in W_T$.
\end{proposition}
We emphasis that, in order to prove null controllability result for  \eqref{sys-memory1} (see Theorem \ref{Thm_null_Control_memo_2}), we only need existence and uniqueness in the case $y_0 \in L^2(0,1)$.

\begin{proof}
The proof of this Proposition is a consequence of \cite[Theorem 1.1]{Grasselli1991}. 

First of all, we transform \eqref{sys-memory1} into the following Cauchy problem
\begin{equation}\label{Cauch_problem}
\left\{
\begin{array}{ll}
\displaystyle y'(t) + Ay(t) = \int\limits_0^t k(t,s, y(s)) \, ds + f(t), \quad t \in (0, T), \\
y(0)= y_{0}, 
\end{array}
\right.
\end{equation}
where
$$ A y(t) := - (a y_x(t))_x, \quad f(t) := 1_{\omega} u(t), \quad \text{for a.e. } t \in (0, T), $$
$$ k(t,s,y(s)) : =  b(t,s, \cdot)  y(s), \quad \text{for a.e. } (t,s) \in (0, T)^2.$$

Next, we are going to check that \eqref{Cauch_problem} satisfies the assumptions in the aforementioned Theorem.
To this aim, let $H_a^{-1}(0,1)$ be the dual space of $H_a^1(0,1)$ with respect to the pivot space
$L^2(0,1)$, endowed with the natural norm
\begin{equation*}
\| z \|_{H_a^{-1}}:=  \sup_{\| y \|_{H_a^{1}}=1} \langle z, y \rangle_{H_a^{-1}, H_a^{1}}.
\end{equation*}
Observe that
$$ \langle Ay, z \rangle_{H_a^{-1}, H_a^{1}} = \int_0^1 a  y_x z_x \, dx, \quad \forall  z \in H_a^1(0,1),$$
$$ \langle k(t,s, y), z \rangle_{H_a^{-1}, H_a^{1}} = \int_0^1 b(t,s,x) y z \, dx, \quad \text{for a.e.} \, (t,s) \in (0, T)^2, \quad \forall z \in H_a^1(0,1), $$
for any $ y \in H_a^1(0,1)$.

Hence, one can check easily that the operators $A$ and $k$ satisfy the following properties:
\begin{enumerate}
\item [$(a)$] there exists a positive constant $C$ such that   $ \|Ay\|_{H_a^{-1}} \leq C  \|y\|_{H_a^1},\quad \forall y \in H_a^1(0,1);$
\item [$(b)$] there exists a positive constant $C$ such that $$ \|A y_1 - A y_2\|_{H_a^{-1}} \leq C  \| y_1 - y_2 \|_{H_a^1}, \quad \text{for any} \; y_1 , y_2 \in H_a^1(0,1);$$
\item [$(c)$] $\exists \;\gamma >0 \text{ and }  \lambda > 0 $ such that $$ \langle A y_1 - A y_2, y_1 - y_2 \rangle_{H_a^{-1}, H_a^{1}} + \lambda \| y_1 - y_2\|_{L^2(0,1)}^2 \geq \gamma  \|y_1 - y_2\|_{H_a^1}^2,$$
for any  $ y_1, y_2 \in H_a^1(0,1);$
\item [$(d)$] there exists a function 
$\beta: (0, T)^2 \mapsto \mathbb{R}^+$ such that $$ \| (k(t,s,y_1)- k(t,s,y_2)) \|_{H_a^{-1}} \leq \beta(t,s) \|y_1 - y_2\|_{H_a^1}, \quad \text{for a.e. } (t,s) \in (0, T)^2, $$ for any  $ y_1, y_2 \in H_a^1(0,1).$

Besides $\beta$ is explicitly given by
$$\beta(t,s) :=\|b(t,s,\cdot)\|_{L^{\infty}(0,1)}, \quad \text{for a.e. } (t,s) \in (0, T)^2.$$
Then, taking into account the fact that $b \in L^{\infty} ((0,T) \times Q)$, $f \in L^2(Q)$ and in view of \cite[Remark 1.2, 1.3]{Grasselli1991}, 
we infer that all the assumptions of \cite[Theorem 1.1]{Grasselli1991} are fulfilled. 
Consequently, the problem \eqref{Cauch_problem} has a unique solution 
 $$ y \in L^2(0, T; H_a^1(0,1))\cap L^{\infty}(0, T; L^2(0,1))$$ 
 with $\quad y_t \in L^2(0, T; H_a^{-1}(0,1)).$
 
Moreover, by Aubin Lions Theorem we also have 
 $$ y \in C([0, T]; L^2(0,1)) .$$
Thus \eqref{Class_W_T_memory} is proved.
\end{enumerate}
\end{proof}
\section{Carleman estimates} \label{sect_carleman_estimate}
The goal of this section is to establish a suitable Carleman estimates for the following adjoint parabolic system 
\begin{equation}\label{sys-adj-nonhom}
\left\{
\begin{array}{lll}
\displaystyle - v_t - (a(x)  v_x )_x = g &  & (t, x) \in Q, \\
v(t,1)= 0 , & & t \in (0, T),  \\
\begin{cases}
& v(t, 0) = 0,   \qquad   \quad (WD), \\
& (a v_x )  (t, 0)= 0 , \quad (SD), \\
\end{cases}  & & t \in (0, T),\\
v(T,x)= v_{T}(x),  & & x \in  (0, 1),
\end{array}
\right.
\end{equation}
where $v_T \in L^2(0,1)$ and $g\in L^2(Q)$.

As a first step, we introduce the following weight functions 
\begin{equation}\label{weightfunc_deg}
\begin{aligned}
&\psi(x):= \gamma \Big(\int_{0}^{x}\frac{y}{a(y)}\,dy - d\Big),\quad \theta(t):=\frac{1}{\big[t(T-t)\big]^4},\\
&\varphi(t,x):=\theta(t)\psi(x),
\end{aligned}
\end{equation} 
Now, let $\tilde{\omega}$ be an arbitrary open subset of $\omega$ and $ \rho \in C^2([0, 1])$ be such that
$$ \rho>0, \; \text{in} \; (0, 1), \quad \rho(0)=\rho(1)=0 \quad \text{and} \quad \rho_x \neq 0, \; \text{in} \; [0, 1]\backslash \tilde{\omega}$$
and define 
\begin{equation}\label{weightfunc_nondeg}
 \Psi(x) := e^{\lambda \rho(x)} - e^{2\lambda \|\rho\|_{\infty}}, \quad \eta(t,x) := \theta(t) \Psi(x).
\end{equation}
The parameters $\lambda$, $d$ and $\gamma$ are positive constant satisfy
\begin{equation}\label{condition_d_gamma}
d>d^\star:=\displaystyle\int_0^1 \frac{y}{a(y)}\,dy, \qquad \gamma >  \frac{e^{2\lambda \|\rho\|_{\infty}}}{(d - d^{\ast})} \end{equation}
and to be specified later on. It clearly follows from \eqref{condition_d_gamma} that
\begin{equation}\label{caract_of_psi}
-\gamma d \leq \psi(x) < 0, \quad \text{for all} \; x \in [0, 1],
\end{equation}
\begin{equation}\label{compar_varphi_eta}
\psi(x) \leq \Psi(x), \;\, \text{for all} \; x \in [0, 1], \;\, \varphi(t,x) \leq \eta(t,x), \;\, \text{for all} \;(t,x) \in Q
\end{equation}
Moreover, we readily have from the definition of the function $\theta$ that 
\begin{equation}\label{bound_theta}
| \theta'(t)| \leq C \theta^{\frac{3}{2}}(t), \;\, \forall \, t \in [0, T], \;\, \text{and} \quad \theta(t) \rightarrow + \infty, \quad \text{as} \quad t \rightarrow 0^-, T^+.    
\end{equation}
\begin{remark}\label{remrk_on_theta}
All the results stated in this paper remains true when the expression of the weighted time function 
$\theta(t):=\frac{1}{[t(T-t)]^4}$ is replaced by $\theta(t):=\frac{1}{[t(T-t)]^2}$. We refer to  \cite[Remark 1]{Hajjaj2013} for a discussion on this fact. 

Moreover, if the problem is considered in $[t_0, T]$ with $t_0>0$, these results still hold considering $\theta(t):=\frac{1}{[(t- t_0)(T-t)]^4}$ or $\theta(t):=\frac{1}{[(t- t_0)(T-t)]^2}$.
\end{remark}
We also remind the following Hardy Poincar\'e inequality, which turns out to be a fundamental tool in proving Carleman estimates:
\begin{proposition}{\cite[Proposition 2.1]{Alabau2006}} \label{prop_hardy}
There is a positive constant $C_{HP}$ such that, for every $y\in H_a^1(0,1)$, the following inequality holds
\begin{equation}\label{harpoideginequality}
\int_0^1\frac{a(x)}{x^2}y^2(x)\,dx\leq C_{HP}\int_0^1a(x)|y_x(x)|^2\,dx.
\end{equation}
\end{proposition}
Then, we have the following Carleman estimate:
\begin{theorem}\label{thm-carleman-1}
Let $T>0$. There exist two positive constants $C$ and $s_0$, such that the solution $v \in Z_T$ of \eqref{sys-adj-nonhom} satisfies
\begin{align}\label{Carl_estimate-1}
\int\!\!\!\!\!\int_{Q} \Big(s \theta  a(x) v_{x}^{2} &+ s^{3} \theta^3 \frac{x ^2}{ a(x)} v^{2}\Big) e^{2s\varphi}\,dx\,dt \notag \\ &\leq C\Big(\int\!\!\!\!\!\int_{Q} g^2 e^{2s\eta}\,dx\,dt  +  \int\!\!\!\!\!\int_{Q_{\omega}} s^3 \theta^3 v^2 e^{2s\eta} \,dx dt\Big)
\end{align}
for all $s \geq s_0$. Here $Q_{\omega} = (0,T) \times \omega$.
\end{theorem}
\begin{proof}
For the proof see \cite[Theorem 3.3]{AHMS19}, where this inequality is established for a coupled parabolic system in $(t_0, T) \times (0, 1)$ instead of $Q$. However, this inequality remains true in $(0, T) \times (0, 1)$ with suitable changes. See also \cite[Lemma 2.4]{Carmelo2000}.
\end{proof}

 Theorem \ref{thm-carleman-1} could be used  to prove null controllability for  \eqref{sys-memory1} under the following hypothesis on the memory kernel  $b$:
\begin{equation}\label{hyp-kernel-k0}
e^{\frac{C^{\ast}}{(T-t)^4}} b \in L^{\infty}((0, T) \times Q)
\end{equation}
for some constant $C^{\ast}>0$. However, we emphasize that, our objective is to provide null controllability for the memory equation \eqref{sys-memory1} for more general memory kernel  $b$. In this purpose, as a first step, we are going to extend the Carleman inequality proved in the previous Theorem in the following way.

\begin{theorem}\label{thm-carleman-2}
Let $k\geq 0$. Then, there exist two positive constants $C$ and $s_0$, such that the solution  $v \in Z_T$ of \eqref{sys-adj-nonhom} satisfies, 
\begin{align}\label{Carl_estimate-2}
\int\!\!\!\!\!\int_{Q} \Big((s\theta)^{1+k}  & a(x) v_{x}^{2} + (s \theta)^{(3+k)} \frac{x ^2}{ a(x)} v^{2}\Big) e^{2s\varphi}\,dx\,dt \notag \\
&\leq C\Big(\int\!\!\!\!\!\int_{Q} (s \theta)^k g^2 e^{2s\eta}\,dx\,dt  +  \int\!\!\!\!\!\int_{Q_{\omega}} (s\theta)^{k+3} v^2 e^{2s\eta} \,dx dt\Big)
\end{align}
for all $s \geq s_0$.
\end{theorem}
\begin{proof}
Let $\omega_2 = (x_1, x_2)$ and $\omega_1$ be two arbitrary subintervals of $\omega$ such that  $\omega_2 \Subset \omega_1$
and consider a smooth cut-off function $\chi \in C^\infty([0, 1])$ such that 

\begin{align*}
0\leq \chi \leq 1, \qquad \chi(x):=
\left\{
\begin{array}{lll}
1,   & \text{for} \; x\in[0,x_1],\\
0 , & \text{for} \; x\in [x_2, 1].
\end{array}
\right. 
\end{align*}
Then, thanks to \cite[Proposition 3.4]{Fadili}, the solution of \eqref{sys-adj-nonhom} satisfies
\begin{align}\label{Carl_estimate_deg_chi}
\int\!\!\!\!\!\int_{Q}& \Big((s\theta)^{1+k} \chi^2  a(x) v_{x}^{2} + (s \theta)^{(3+k)} \chi^2 \frac{x^2}{ a(x)} v^{2}\Big) e^{2s\varphi}\,dx\,dt \notag \\
&\leq C\Big(\int\!\!\!\!\!\int_{Q} (s \theta)^k \chi^2 g^2 e^{2s\varphi}\,dx\,dt+  \int\!\!\!\!\!\int_{Q_{\omega_1}} (s\theta)^{k} \Big(f^2+ (s\theta)^{2} v^2 \Big) e^{2s\varphi} \,dx dt\Big).
\end{align} 
On the other hand, let $\zeta := 1 - \chi$, it follows from \cite[Proposition 3.5]{Fadili} that  
\begin{align}\label{Carl_estimate_nondeg_zeta}
\int\!\!\!\!\!\int_{Q} \Big((s\theta)^{1+k} \zeta^2 & a(x) v_{x}^{2} + (s \theta)^{(3+k)} \zeta^2 \frac{x^2}{ a(x)} v^{2}\Big) e^{2s\eta}\,dx\,dt \notag \\
&\leq C\Big(\int\!\!\!\!\!\int_{Q} (s \theta)^k \zeta^2 g^2 e^{2s\eta}\,dx\,dt+  \int\!\!\!\!\!\int_{Q_{\omega_1}} (s\theta)^{k+3} v^2 e^{2s\eta} \,dx dt\Big).
\end{align} 
Therefore, using \eqref{compar_varphi_eta}, \eqref{Carl_estimate_deg_chi}, \eqref{Carl_estimate_nondeg_zeta} and the fact that $\frac{1}{2} \leq \chi^2 +\zeta^2 \leq 1$ there holds
\begin{align*}
\int\!\!\!\!\!\int_{Q} \Big((s\theta)^{1+k}  & a(x) v_{x}^{2} + (s \theta)^{(3+k)} \frac{x ^2}{ a(x)} v^{2}\Big) e^{2s\varphi}\,dx\,dt \notag \\
&\leq C\Big(\int\!\!\!\!\!\int_{Q} (s \theta)^k g^2 e^{2s\eta}\,dx\,dt  +  \int\!\!\!\!\!\int_{Q_{\omega}} (s\theta)^{k+3} v^2 e^{2s\eta} \,dx dt\Big)
\end{align*}
which concludes Theorem \ref{thm-carleman-2}. 
\end{proof}
Next, by \eqref{Carl_estimate-2}, we are going to derive a new modified Carleman inequality, that is an estimate with a weight time function exploding only at the final time $t=T$. This choice is done recalling the technique developed by A.V. Fursikov and O.Y. Imanuvilov in \cite{FI1996} in the context of uniformly parabolic equations. In our setting, this new weight allows us to derive a null controllability result for system \eqref{sys-memory1}  imposing a restriction on the kernel $b$ only at the final time $t=T$ (see  \eqref{hypoth_kernel}). To this end, let us introduce the following weight functions:
\begin{equation*}
\beta(t):=
\left\{
\begin{array}{lll}
\theta(\frac{T}{2}) = \big(\frac{2}{T}\big)^8,   & \text{for} \; t\in\big[0,\frac{T}{2}\big],\\
\theta(t) , & \text{for} \; t\in \big[\frac{T}{2}, T\big],
\end{array}
\right. \Phi(t,x) =  \beta(t)\psi(x),\quad \sigma(t,x): = \beta(t)\Psi(x)
\end{equation*}
and
\begin{equation}\label{modif_weight_funct}
\begin{aligned}
& \widehat{\Phi}(t) : = \max\limits_{x\in[0,1]} \Phi(t,x)= \gamma (d^{\ast} -d) \beta(t), \\
& {\Phi^*(t)} :=  \min\limits_{x\in[0,1]} \Phi(t,x)= -\gamma d \beta(t).
\end{aligned}
\end{equation}
In view of \eqref{compar_varphi_eta}, we can see that the weight functions $\Phi$ and $\sigma$  satisfy the following inequality which is needed in what follows
\begin{equation}\label{compar_modif_weight}
 \Phi(t,x) \leq \sigma(t,x), \quad \forall (t,x)  \, \in Q.
\end{equation} 

Now, we are ready to state the following modified Carleman estimate, which reveals to be a major tool to obtain the null controllability result given in Theorem \ref{Thm_null_Control_memo_2}.
 \begin{lemma}\label{Lemma_modif_Carl_estimate}
Let $k\geq 0$. Then, there exists two positive constants $C$ and $s_0$ such that every solution $ v \in Z_T$ of system \eqref{sys-adj-nonhom} satisfies
\begin{align}\label{modif_Carl_estimate}
&s^k e^{2s\widehat{\Phi}(0)}\| v(0)\|_{L^2(0,1)}^2 + \int\!\!\!\!\!\int_{Q} (s \beta)^{k} v^2  e^{2s\Phi}\,dx\,dt \notag \\
&\qquad \leq C e^{2s[\widehat{\Phi}(0) - \Phi^*(\frac{5T}{8})]} \Big(\int\!\!\!\!\!\int_{Q} (s\beta)^{k} g^2 e^{2s\sigma}\,dx\,dt +  \int\!\!\!\!\!\int_{Q_{\omega}} (s\beta)^{k+3} v^2 e^{2s\sigma} \,dx dt \Big)
\end{align}
for all $s \geq s_0$. 
\end{lemma}
\begin{proof}
Let $ \xi \in C^{\infty}([0, T])$ be a cut-off function such that
\begin{align}\label{Xi}
0\leq \xi \leq 1, \qquad \xi(t):=
\left\{
\begin{array}{lll}
1,   & \text{for} \; t\in\Big[0,\dfrac{T}{2}\Big],\\
0 , & \text{for} \; t\in \Big[\dfrac{5T}{8}, T\Big]
\end{array}
\right. 
\end{align}
and define $ w = \tilde{\xi} v$ ,  where $ \tilde{\xi} = \beta^{\frac{k}{2}} \xi  e^{s\widehat{\Phi}(0)} $ and $v$ solves \eqref{sys-adj-nonhom}.

Hence $w$ satisfies
\begin{equation}\label{sys-adj-nonhom-w}
\left\{
\begin{array}{lll}
\displaystyle - w_t - (a(x)  w_x )_x = - \tilde{\xi}' v + \tilde{\xi} g,  &  & (t, x) \in Q, \\
w(t,1)= 0 , & & t \in (0, T),  \\
\begin{cases}
& w(t, 0) = 0,   \qquad   \quad (WD), \\
& (a w_x )  (t, 0)= 0 , \quad (SD), \\
\end{cases}  & & t \in (0, T),\\
w(T,x)= 0,  & & x \in  (0, 1).
\end{array}
\right.
\end{equation}
Then, by \eqref{energy-nonhom1} applied to the above system, one can see that
\begin{equation}\label{energy-w}
\| w(0)\|_{L^2(0,1)}^2 +\| w\|_{L^2(Q)}^2  \leq C \int\!\!\!\!\!\int_{Q} (- \tilde{\xi}' v + \tilde{\xi} g)^2  \,dx\,dt
\end{equation}
for some constants $C>0$.

We estimate from below the two terms on the left hand side of \eqref{energy-w} in the following way:
\begin{equation}\label{left-enrgy-1}
\| w(0)\|_{L^2(0,1)}^2 =   \| \beta^{\frac{k}{2}}(0) \xi(0) e^{s\widehat{\Phi}(0)} v(0)\|_{L^2(0,1)}^2 = \Big(\dfrac{2}{T}\Big)^{8k} \| e^{s\widehat{\Phi}(0)} v(0)\|_{L^2(0,1)}^2 
\end{equation}
and
\begin{equation}\label{left-enrgy-w-2}
\| w\|_{L^2(Q)}^2  =   \int_{0}^{\frac{5T}{8}}\!\!\!\!\int_0^1 \beta^{k}\xi^2  e^{2s\widehat{\Phi}(0)} v^2 \,dx\,dt \geq \int_{0}^{\frac{T}{2}}\!\!\!\!\int_0^1 \beta^{k} v^2 e^{2s\Phi}  \,dx\,dt 
\end{equation}
since $ \Phi \leq \widehat{\Phi}(0) \; \text{in} \, Q$.

Concerning the right hand side of \eqref{energy-w}, we have
\begin{align}\label{right-enrgy-w}
&\int\!\!\!\!\!\int_{Q} (- \tilde{\xi}' v + \tilde{\xi} g)^2  \,dx\,dt \notag\\ 
& =  \int\!\!\!\!\!\int_{Q}\Big[ \Big(- \frac{k}{2} \beta' \beta^{\frac{k}{2}-1} \xi - \beta^{\frac{k}{2}} \xi' \Big) e^{s\widehat{\Phi}(0)} v + \beta^{\frac{k}{2}} \xi e^{s\widehat{\Phi}(0)} g\Big]^2  \,dx\,dt \notag \\
& \leq C \bigg(\int\!\!\!\!\!\int_{Q}  (\beta')^2 \beta^{k-2}  \xi^2 e^{2s\widehat{\Phi}(0)} v^2   \,dx\,dt \notag\\
& \quad + \int\!\!\!\!\!\int_{Q}  \beta^{k} (\xi')^2  e^{2s\widehat{\Phi}(0)} v^2  \,dx\,dt + \int\!\!\!\!\!\int_{Q}  \beta^k \xi^2 e^{2s\widehat{\Phi}(0)} g^2 \,dx\,dt \bigg).
\end{align}
Observing that $\beta'= 0 \; \text{in} \; [0, T/2]$, $\beta= \theta \; \text{in} \; [T/2, T]$ and using \eqref{bound_theta}, the fact that $ \supp \xi \subset [0, 5T/8]$ and $ \supp \xi' \subset [T/2, 5T/8]$, it follows that
\begin{equation}\label{right-enrgy-w-1}
e^{2s\widehat{\Phi}(0)} \int\!\!\!\!\!\int_{Q} (\beta')^2 \beta^{k-2}  \xi^2 v^2   \,dx\,dt \leq C e^{2s\widehat{\Phi}(0)} \int_{\frac{T}{2}}^{\frac{5T}{8}}\!\!\!\!\int_0^1 \beta^{k+1}   v^2   \,dx\,dt
\end{equation}
and 
\begin{align}\label{right-enrgy-w-2}
e^{2s\widehat{\Phi}(0)} \int\!\!\!\!\!\int_{Q}  \beta^{k} (\xi')^2 v^2  \,dx\,dt
& \leq C e^{2s\widehat{\Phi}(0)} \int_{\frac{T}{2}}^{\frac{5T}{8}}\!\!\!\!\int_0^1 \beta^{k}   v^2   \,dx\,dt\notag\\ 
& \leq C  e^{2s\widehat{\Phi}(0)} \int_{\frac{T}{2}}^{\frac{5T}{8}}\!\!\!\!\int_0^1 \beta^{k+1}   v^2   \,dx\,dt.
\end{align} 
Hence, by the estimates \eqref{energy-w}-\eqref{right-enrgy-w-2}, we find that
\begin{align}\label{estim-obs-v-1}
s^k \|e^{s\widehat{\Phi}(0)}  v(0) &\|_{L^2(0,1)}^2  + \int_{0}^{\frac{T}{2}}\!\!\!\!\int_0^1 (s\beta)^{k} v^2 e^{2s\Phi}  \,dx\,dt \notag \\
& \leq C \Big(s^k \int_{\frac{T}{2}}^{\frac{5T}{8}}\!\!\!\!\int_0^1 s^k \beta^{k+1}   v^2  e^{2s\widehat{\Phi}(0)}  \,dx\,dt + \int_{0}^{\frac{5T}{8}}\!\!\!\!\int_0^1  (s\beta)^k e^{2s\widehat{\Phi}(0)} g^2 \,dx\,dt \Big).
\end{align}
Now, let us deal with the first term in the right-hand side of \eqref{estim-obs-v-1}.

First, using  the fact that $\beta =\theta$ and $\Phi= \varphi \; \text{in} \; [T/2, T]$, one has 
\begin{align}\label{equality_estim_Phi_varphi}
\int_{\frac{T}{2}}^{\frac{5T}{8}}\!\!\!\!\int_0^1 s^k \beta^{k+1} v^2  e^{2s\Phi}\,dx\,dt & =  \int_{\frac{T}{2}}^{\frac{5T}{8}}\!\!\!\!\int_0^1 s^k \theta^{k+1} v^2  e^{2s\varphi}\,dx\,dt.
\end{align}
Then, applying Young's inequality as in \cite{Alabau2006}, we see that
\begin{align}\label{Young_vexp}
\int_0^1 v^2  e^{2s\varphi}\,dx& = \int_0^1  \bigg(\Big(\frac{a(x)}{x^2}\Big)^{\frac{1}{3}} v^2 e^{2s\varphi}\bigg)^{\frac{3}{4}}
\bigg(\frac{x^2}{a(x)}v^2 e^{2s\varphi} \bigg)^{\frac{1}{4}}\,dx\nonumber \\
&\leq \frac{3}{4} \int_0^1 \Big(\frac{a(x)}{x^2} \Big)^{\frac{1}{3}} v^2  e^{2s\varphi}\,dx+
\frac{1}{4} \int_0^1  \frac{x^2}{a(x)} v^2 e^{2s\varphi}\,dx.
\end{align}
Let $p(x)=x^{4/3}a^{1/3}$, then since the function $x \mapsto \displaystyle \frac{x^2}{a}$ is nondecreasing on $(0,1)$ one has,
$$p(x)= a \Big(\frac{x^2}{a}\Big)^{\frac{2}{3}}\leq C a(x).$$
Then, applying the Hardy-Poincar\'e inequality \eqref{harpoideginequality} to $v e^{s\varphi} $, we get
\begin{align}\label{hardy-vexp}
\int_{0}^{1} \frac{a^{1/3}}{x^{2/3}} (v e^{s\varphi})^2\,dx & = \int_{0}^{1} \frac{p(x)}{x^{2}} (v e^{s\varphi})^2\,dx \notag \\ 
&\leq C \int_{0}^{1} p(x) (v e^{s\varphi})_x^2\,dx
\leq C \int_{0}^{1} a(x) (v e^{s\varphi})_x^2\,dx.
\end{align}
Using the definition of $\varphi$ (see \eqref{weightfunc_deg}), it follows that
\begin{align}\label{est_first_term_young}
\int_{0}^{1} \frac{a^{1/3}}{x^{2/3}} (v e^{s\varphi})^2\,dx &\leq C \int_{0}^{1} a(x) \big(v_x + s \varphi_x v)^2 e^{2s\varphi}\,dx \notag\\
& \leq C \int_{0}^{1}  \Big(a(x) v_x^2 + s^2 \theta^2 \frac{x^2}{a(x)} v^2 \Big) e^{2s\varphi}\,dx
.
\end{align}
By \eqref{Young_vexp} and \eqref{est_first_term_young}, we obtain 
\begin{align}\label{esti-vexp1}
\int_0^1 v^2  e^{2s\varphi}\,dx & \leq C \int_{0}^{1}  \Big(a(x) v_x^2 + s^2 \theta^2 \frac{x^2}{a(x)} v^2 \Big) e^{2s\varphi}\,dx.
\end{align}
Hence, from \eqref{equality_estim_Phi_varphi} and \eqref{esti-vexp1}, we get that
\begin{align*}
\int_{\frac{T}{2}}^{\frac{5T}{8}}\!\!\!\!\int_0^1 & s^k \beta^{k+1} v^2  e^{2s\Phi}\,dx\,dt \notag\\
& \leq C \int_{\frac{T}{2}}^{\frac{5T}{8}}\!\!\!\!\int_0^1 s^k \theta^{k+1} \Big(a(x) v_x^2 + s^2 \theta^{2} \frac{x^2}{a(x)} v^2 \Big) e^{2s\varphi}\,dxdt.
\end{align*}
Thus, applying Carleman inequality \eqref{Carl_estimate-2}, one has
\begin{align}\label{est-right-obs-1}
s^k \int_{\frac{T}{2}}^{\frac{5T}{8}} \!\!\!\!\int_0^1 \beta^{k+1} v^2  e^{2s\Phi}\,dx\,dt & \leq C \int_{\frac{T}{2}}^{\frac{5T}{8}}\!\!\!\!\int_0^1  \Big(s^k \theta^{k+1}a(x) v_x^2 + s^{k+2} \theta^{k+3} \frac{x^2}{a(x)} v^2 \Big) e^{2s\varphi}\,dxdt \notag\\
& \leq C \int\!\!\!\!\!\int_{Q}  \Big((s\theta)^{k+1}a(x) v_x^2 + (s\theta)^{k+3} \frac{x^2}{a(x)} v^2 \Big) e^{2s\varphi}\,dxdt\notag\\
&\leq C\Big(\int\!\!\!\!\!\int_{Q} (s \theta)^k g^2 e^{2s\eta}\,dx\,dt  +  \int\!\!\!\!\!\int_{Q_{\omega}} (s\theta)^{k+3} v^2 e^{2s\eta} \,dx dt\Big).
\end{align}
Now observe that
\begin{equation}\label{inequality_Phi}
\Phi^*\big(\frac{5T}{8}\big)\leq \Phi, \quad \text{in} \; \Big[0, \frac{5T}{8}\Big]\times \big[0,1\big].
\end{equation}
Therefore
\begin{align}\label{estim-obs-v-1_right1}
 &s^k\int_{\frac{T}{2}}^{\frac{5T}{8}}\!\!\!\!\int_0^1   \beta^{k+1} v^2  e^{2s\widehat{\Phi}(0)} \,dx\,dt \notag = s^k\int_{\frac{T}{2}}^{\frac{5T}{8}}\!\!\!\!\int_0^1  \beta^{k+1} v^2 e^{2s [\widehat{\Phi}(0)-\Phi]} e^{2s\Phi}\,dx\,dt\notag\\
& \leq e^{2s [\widehat{\Phi}(0)-\Phi^*(\frac{5T}{8})]} s^k\int_{\frac{T}{2}}^{\frac{5T}{8}}\!\!\!\!\int_0^1  \beta^{k+1} v^2  e^{2s\Phi}\,dx\,dt\notag\\
&\leq Ce^{2s [\widehat{\Phi}(0)-\Phi^*(\frac{5T}{8})]} \Big(\int\!\!\!\!\!\int_{Q} (s \theta)^k g^2 e^{2s\eta}\,dx\,dt  + \int\!\!\!\!\!\int_{Q_{\omega}} (s\theta)^{k+3} v^2 e^{2s\eta} \,dx dt\Big)
\end{align}
for $s$ large enough.

Moreover, in view of \eqref{compar_modif_weight} and \eqref{inequality_Phi}, the second term in the right hand side of \eqref{estim-obs-v-1} reads as
\begin{align*}
\int_{0}^{\frac{5T}{8}}\!\!\!\!\int_0^1  (s\beta)^k e^{2s\widehat{\Phi}(0)} g^2 \,dx\,dt =  \int_{0}^{\frac{5T}{8}}\!\!\!\!\int_0^1  (s\beta)^k  e^{2s [\widehat{\Phi}(0)-\Phi]} e^{2s\Phi} g^2 \,dx\,dt\notag\\
\leq e^{2s[\widehat{\Phi}(0)-\Phi^*(\frac{5T}{8})]} \int_{0}^{\frac{5T}{8}}\!\!\!\!\int_0^1  (s\beta)^k  e^{2s\sigma} g^2 \,dx\,dt.
\end{align*}
Combining this last inequality with \eqref{estim-obs-v-1} and \eqref{estim-obs-v-1_right1}, it follows that
\begin{align}\label{estim-obs-v-2}
& s^k \|e^{s\widehat{\Phi}(0)} v(0)\|_{L^2(0,1)}^2  + \int_{0}^{\frac{T}{2}}\!\!\!\!\int_0^1 (s \beta)^{k} v^2 e^{2s\Phi}  \,dx\,dt \notag \\
&\qquad  \leq Ce^{2s [\widehat{\Phi}(0)-\Phi^*(\frac{5T}{8})]} \Big(\int\!\!\!\!\!\int_{Q} (s \theta)^k g^2 e^{2s\eta}\,dx\,dt + \int\!\!\!\!\!\int_{Q_{\omega}} (s\theta)^{k+3} v^2 e^{2s\eta} \,dx dt  \notag \\ 
&  \qquad \qquad \qquad  \qquad \quad \qquad + \int_{0}^{\frac{5T}{8}}\!\!\!\!\int_0^1  (s \beta)^k g^2 e^{2s\sigma}\,dx\,dt\Big).
\end{align}

On the other hand, proceeding as in \eqref{est-right-obs-1}, we also obtain
\begin{align}\label{estim-obs-v-3}
\int_{\frac{T}{2}}^T \!\!\!\!\int_0^1 (s \beta)^{k} v^2 e^{2s\Phi}  \,dx\,dt & = \int_{\frac{T}{2}}^T \!\!\!\!\int_0^1 (s \beta)^{k} v^2 e^{2s\varphi}  \,dx\,dt \notag\\
&\leq C
\int_{\frac{T}{2}}^T \!\!\!\!\int_0^1  \Big((s \theta)^{k} a(x) v_x^2 + (s \theta)^{k+2} \frac{x^2}{a(x)} v^2 \Big) e^{2s\varphi}\,dxdt \notag\\
&\leq C\Big(\int\!\!\!\!\!\int_{Q} (s \theta)^k g^2 e^{2s\eta}\,dx\,dt +  \int\!\!\!\!\!\int_{Q_{\omega}} (s\theta)^{k+3} v^2 e^{2s\eta} \,dx dt\Big).
\end{align}
Note that, since the function $ s \rightarrow s^k e^{cs},$ with $k \geq 0$ and $c<0$, is nonincreasing for larger values of $s$, then, from the fact that $\beta \leq \theta$ in $(0, T)$ we get that,  $$ (s \theta)^k e^{2s\eta} = (s \theta)^k  e^{2s\Psi(x) \theta(t)} \leq (s \beta)^k  e^{2s\Psi(x) \beta(t)} = (s\beta)^{k} e^{2s\sigma}, \quad \text{in} \; Q,$$
for $s$ large enough, where we recall that $\Psi$  is the weight function given in \eqref{weightfunc_nondeg}. 

Finally, combining this fact with the estimates \eqref{estim-obs-v-2} and \eqref{estim-obs-v-3}, we deduce that
\begin{align*}
& s^k \|e^{s\widehat{\Phi}(0)} v(0)\|_{L^2(0,1)}^2  + \int\!\!\!\!\!\int_{Q} (s \beta)^{k} v^2 e^{2s\Phi}  \,dx\,dt \notag \\
&\qquad  \leq Ce^{2s [\widehat{\Phi}(0)-\Phi^*(\frac{5T}{8})]} \Big(\int\!\!\!\!\!\int_{Q} (s\beta)^k g^2 e^{2s\sigma} \,dx\,dt  +  \int\!\!\!\!\!\int_{Q_{\omega}} (s\beta)^{k+3} v^2 e^{2s\sigma} \,dx dt\Big).
\end{align*}
This ends the proof of Lemma \ref{Lemma_modif_Carl_estimate}.
\end{proof}
\section{Null controllability for system \eqref{sys-nonhom}}\label{sect_null_control_nonhomog}
In this section we will apply the Carleman estimates established in Section \ref{sect_carleman_estimate} to deduce the null controllability result for the nonhomogeneous problem \eqref{sys-nonhom}. To this aim, following the arguments presented in \cite{FI1996, TaoGao16}, we introduce, for all $k \geq 0$, the following weighted space

$$ E_{s,k} = \big\{y\in Z_T : \quad (s\beta)^{-k/2} e^{-s\sigma} y \in L^2(Q)\big\}$$
endowed with the associated norm
\begin{align*}
\| y \|_{E_{s,k}}^2:=  \int\!\!\!\!\!\int_{Q} (s\beta)^{-k} e^{-2 s \sigma} y^2 \, dx dt.
\end{align*}
Observe that, if we consider $y$ in $E_{s,k}$, then $y$ is continuous in time and satisfies
$$\int\!\!\!\!\!\int_{Q} (s\beta)^{-k} e^{-2 s \sigma} y^2 \, dx dt<+\infty,$$
thus, from the definition of $\sigma$, in particular the fact that $\sigma<0$, we have that $$ y(T,\cdot) = 0 \quad \text{in} \; (0,1).$$

In the following, we denote by $s_0$ the parameter defined in Lemma \ref{Lemma_modif_Carl_estimate}.

Then, we are going to prove the following:
\begin{theorem}\label{Thm_null_Control_nonhom}
Let $T > 0$ and $k \geq 0$. Assume $ (s\beta)^{-k/2} e^{-s\Phi} f \in L^2(Q)$ with $s \geq s_0$. Then, for any $y_0 \in H_a^1(0,1)$,
there exists $u \in L^2(Q)$ such that the associated solution $y$ of system \eqref{sys-nonhom} belongs to $E_{s,k}.$

Moreover,  there exists a positive constant $C$ such that the couple $(y, u)$ satisfies
\begin{align}\label{estim_solut_control}
& \int\!\!\!\!\!\int_{Q}  (s \beta)^{-k}  e^{-2s\sigma} y^2  \,dx\,dt +  \int\!\!\!\!\!\int_{Q_{\omega}} (s\beta)^{-(k+3)} e^{-2s\sigma} u^2 \,dx dt \notag \\
&\qquad  \leq Ce^{2s [\widehat{\Phi}(0)-\Phi^*(\frac{5T}{8})]} \Big(\int\!\!\!\!\!\int_{Q} (s\beta)^{-k} e^{-2s\Phi} f^2 \,dx\,dt+ s^{-k} \int_0^1 e^{-2s\widehat{\Phi}(0)} y_0^2\,dx   \Big).
\end{align}

\end{theorem}
\begin{proof}
The proof of this Theorem is inspired by \cite{FI1996,TaoGao16}. First of all,  consider the following functional:
$$ J(y,u) = \int\!\!\!\!\!\int_{Q}  (s \beta)^{-k}  e^{-2s\sigma} y^2  \,dx\,dt +  \int\!\!\!\!\!\int_{Q_{\omega}} (s\beta)^{-(k+3)} e^{-2s\sigma} u^2 \,dx dt$$
where $(y, u)$ satisfies system \eqref{sys-nonhom} with $u \in L^2(Q)$ and
\begin{equation}\label{null_state_y_J}
y(T, \cdot) = 0 \qquad \text{in} \; (0, 1).
\end{equation}
By classical arguments (see for instance \cite{Lions71, Lions83}), one can show that $J$ attains its minimizer at a unique point say, $(\tilde{y},\tilde{u})$.

We are going to prove the existence of a dual variable $\tilde{z}$ such that 
\begin{align*}
\begin{cases}
\tilde{y}= (s\beta)^k e^{2s\sigma} \mathcal{L}^*\tilde{z} & \text{in} \; Q\\
\tilde{u}= - 1_{\omega}   (s\beta)^{k+3}  e^{2s\sigma} \tilde{z} & \text{in} \; Q
\end{cases}
\end{align*}
where $ \mathcal{L}^* \tilde{z}= - \tilde{z}_t - (a(x)  \tilde{z}_x )_x$ and $\tilde{z}$ satisfies the boundary conditions 
\begin{equation}\label{BConditions}
 \tilde{z}(\cdot,1)= 0 \quad \text{and} \;
\left\{
\begin{array}{lll}
	& \tilde{z}(\cdot, 0) = 0,   \qquad   \quad (WD) \\
	& (a \tilde{z}_x )  (\cdot, 0)= 0 , \quad (SD) \\
\end{array} \qquad \text{on} \; (0,T)
\right.
\end{equation}

Let us define the following linear space
$$X_a = \Big\{ w \in C^{\infty}(\overline{Q}): \quad  w \quad \text{satisfies} \; \eqref{BConditions}
 \Big\}.$$
In addition, we set
 \begin{align}\label{bilin_form}
 \kappa(z,w) = \int\!\!\!\!\!\int_{Q} (s\beta)^k e^{2s\sigma} \mathcal{L}^*z \mathcal{L}^*w \,dx\,dt  +  \int\!\!\!\!\!\int_{Q_{\omega}} (s\beta)^{k+3}  e^{2s\sigma} z w \,dx dt, \qquad \forall \, z, w \in X_a
 \end{align}
 and
 \begin{align}\label{lin_form}
\ell(w) = \int\!\!\!\!\!\int_{Q} f w \,dx\,dt  +  \int_0^1 y_0 w(0)  dx, \qquad \forall \,  w \in X_a,
\end{align} 
where $f, y_0$ are the functions in \eqref{sys-nonhom}.

Observe that Carleman estimate \eqref{modif_Carl_estimate} holds for all $w \in X_a$. In particular, we have
\begin{equation*}
 s^{k} \int_0^1 e^{2s\widehat{\Phi}(0)} w(0)^2\,dx + \int\!\!\!\!\!\int_{Q} (s\beta)^{k} e^{2s\Phi} w^2 \,dx\,dt  \leq  Ce^{2s [\widehat{\Phi}(0)-\Phi^*(\frac{5T}{8})]} \kappa(w,w), \qquad \forall \,  w \in X_a.
\end{equation*}
Now, let us denote by $ \widetilde{X}_a $ the completion of $X_a$ with the norm $ \| w\|_{\widetilde{X}_a} = (\kappa(w,w))^{1/2}$. Thus, $ \widetilde{X}_a $ is a Hilbert space with this norm. 

Clearly, $\kappa$ is a strictly positive, symmetric and continuous  bilinear form in $ \widetilde{X}_a $.

Moreover, in view of the above inequality, one can see that the linear form $\ell$ is continuous in $ \widetilde{X}_a $. Indeed, employing the Cauchy-Schwarz inequality, for all  $w \in \widetilde{X}_a$, we have
\begin{align}\label{bound_ell}
|\ell(w)| &= \int\!\!\!\!\!\int_{Q} f w \,dx\,dt  +  \int_0^1 y_0 w(0)  dt\notag \\
& \leq 
\bigg(\Big(\int\!\!\!\!\!\int_{Q} (s\beta)^{-k} e^{-2s\Phi} f^2 \,dx\,dt \Big)^{1/2} \Big(\int\!\!\!\!\!\int_{Q} (s\beta)^{k} e^{2s\Phi} w^2 \,dx\,dt \bigg)^{1/2} \notag\\
& \qquad + \Big(s^{-k} \int_0^1 e^{-2s\widehat{\Phi}(0)} y_0^2\,dx  \Big)^{1/2}  \Big(s^{k} \int_0^1 e^{2s\widehat{\Phi}(0)} w(0)^2\,dx  \Big)^{1/2} \bigg)\notag \\
& \leq  
\bigg(\Big[\Big(\int\!\!\!\!\!\int_{Q} (s\beta)^{-k} e^{-2s\Phi} f^2 \,dx\,dt \Big)^{1/2} + \Big(s^{-k} \int_0^1 e^{-2s\widehat{\Phi}(0)} y_0^2\,dx  \Big)^{1/2} \Big] \times \notag\\
& \qquad\Big[\Big(\int\!\!\!\!\!\int_{Q} (s\beta)^{k} e^{2s\Phi} w^2 \,dx\,dt \Big)^{1/2}  +  \Big(s^{k} \int_0^1 e^{2s\widehat{\Phi}(0)} w(0)^2\,dx  \Big)^{1/2} \Big]\bigg)\notag\\
& \leq Ce^{s [\widehat{\Phi}(0)-\Phi^*(\frac{5T}{8})]} \Big[\Big(\int\!\!\!\!\!\int_{Q} (s\beta)^{-k} e^{-2s\Phi} f^2 \,dx\,dt \Big)^{1/2}\notag \\
& \qquad \qquad \qquad \qquad + \Big(s^{-k} \int_0^1 e^{-2s\widehat{\Phi}(0)} y_0^2\,dx  \Big)^{1/2} \Big] \| w\|_{\widetilde{X}_a}.
\end{align}
Hence, by Lax-Milgram Theorem, we infer that there exists a unique $\tilde{z} \in \widetilde{X}_a$ such that 
\begin{equation}\label{variational_equat}
\kappa(\tilde{z},w)=\ell(w), \qquad \forall \,  w \in \widetilde{X}_a.
\end{equation}
This fact, together with \eqref{bound_ell}, gives that
\begin{align*}
\kappa(\tilde{z},\tilde{z})=\ell(\tilde{z}) \leq Ce^{s [\widehat{\Phi}(0)-\Phi^*(\frac{5T}{8})]} & \Big[\Big(\int\!\!\!\!\!\int_{Q} (s\beta)^{-k} e^{-2s\Phi} f^2 \,dx\,dt \Big)^{1/2}\notag \\
& + \Big(s^{-k} \int_0^1 e^{-2s\widehat{\Phi}(0)} y_0^2\,dx  \Big)^{1/2} \Big] \| \tilde{z}\|_{\widetilde{X}_a}.
\end{align*}
This implies
\begin{align}\label{bound_z_tilde}
\| \tilde{z} \|_{\widetilde{X}_a} \leq Ce^{s [\widehat{\Phi}(0)-\Phi^*(\frac{5T}{8})]} & \Big[\Big(\int\!\!\!\!\!\int_{Q} (s\beta)^{-k} e^{-2s\Phi} f^2 \,dx\,dt \Big)^{1/2}\notag \\
& + \Big(s^{-k} \int_0^1 e^{-2s\widehat{\Phi}(0)} y_0^2\,dx  \Big)^{1/2} \Big].
\end{align}
Setting
\begin{align}\label{def_tilde_y_u}
\begin{cases}
\tilde{y}= (s\beta)^k e^{2s\sigma} \mathcal{L}^*\tilde{z} \\
\tilde{u}= - 1_{\omega}   (s\beta)^{k+3}  e^{2s\sigma} \tilde{z}
\end{cases}
\end{align}
and using the definition of the bilinear form $\kappa(\cdot, \cdot)$, we can write
\begin{align*}
 \| \tilde{z} \|_{\widetilde{X}_a}^2 & =  \int\!\!\!\!\!\int_{Q}  (s \beta)^{-k}  e^{-2s\sigma} \big((s\beta)^k e^{2s\sigma} \mathcal{L}^*\tilde{z}\big)^2  \,dx\,dt \notag \\
& \quad +  \int\!\!\!\!\!\int_{Q_{\omega}} (s\beta)^{-(k+3)} e^{-2s\sigma} \big(1_{\omega}   (s\beta)^{k+3}  e^{2s\sigma} \tilde{z}\big)^2 \,dx dt \notag \\
& = \int\!\!\!\!\!\int_{Q}  (s \beta)^{-k}  e^{-2s\sigma} \tilde{y}^2  \,dx\,dt +  \int\!\!\!\!\!\int_{Q_{\omega}} (s\beta)^{-(k+3)} e^{-2s\sigma} \tilde{u}^2 \,dx dt
\end{align*}
and, in view of \eqref{bound_z_tilde}, we can deduce
\begin{align}\label{estim_tilde_y_u}
& \int\!\!\!\!\!\int_{Q}  (s \beta)^{-k}  e^{-2s\sigma} \tilde{y}^2  \,dx\,dt +  \int\!\!\!\!\!\int_{Q_{\omega}} (s\beta)^{-(k+3)} e^{-2s\sigma} \tilde{u}^2 \,dx dt \notag \\
&\qquad  \leq Ce^{2s [\widehat{\Phi}(0)-\Phi^*(\frac{5T}{8})]} \Big(\int\!\!\!\!\!\int_{Q} (s\beta)^{-k} e^{-2s\Phi} f^2 \,dx\,dt+ s^{-k} \int_0^1 e^{-2s\widehat{\Phi}(0)} y_0^2\,dx   \Big).
\end{align}
Hence $\tilde y \in E_{s,k}$ and satisfies the inequality \eqref{estim_solut_control}.

In order to complete the proof, it remains  to show that $(\tilde{y}, \tilde{u})$, satisfies the parabolic problem \eqref{sys-nonhom} and the identity \eqref{null_state_y_J}.
First of all, by \eqref{estim_tilde_y_u} it is immediate that $\tilde{y}, \tilde{u} \in L^2(Q)$. 

Moreover, denote by $\hat{y}$ the weak solution of system \eqref{sys-nonhom} associated to the control function $u=\tilde{u}$. Then, $\hat{y}$ also solves this system in the sense of transposition, that is, $\hat{y}$ is the unique function in $ L^2(Q)$ satisfying 
\begin{equation}\label{sol_by_transpo}
\int\!\!\!\!\!\int_{Q} \hat{y}  h \,dx\,dt = \int_0^1 y_0 w(0)  dx + \int\!\!\!\!\!\int_{Q} 1_{\omega} \tilde{u} w \,dx\,dt +  \int\!\!\!\!\!\int_{Q} f w \,dx\,dt , \qquad \forall \, h \in L^2(Q),
\end{equation}
where $w$ is the solution of
\begin{equation*}
\left\{
\begin{array}{lll}
\displaystyle - w_t - (a(x)  w_x )_x = h &  & (t, x) \in Q, \\
w(t,1)= 0 , & & t \in (0, T),  \\
\begin{cases}
& w(t, 0) = 0,   \qquad   \quad (WD), \\
& (a w_x )  (t, 0)= 0 , \quad (SD), \\
\end{cases}  & & t \in (0, T),\\
w(T,x)= 0,  & & x \in  (0, 1).
\end{array}
\right.
\end{equation*}
On the other hand, substituting the expressions of $\tilde{y}$ and $\tilde{u}$, given in \eqref{def_tilde_y_u}, in \eqref{variational_equat}, we obtain
\begin{equation}\label{sol_by_transpo1}
 \int\!\!\!\!\!\int_{Q} \widetilde{y}  h \,dx\,dt - \int\!\!\!\!\!\int_{Q_{\omega}} 1_{\omega} \tilde{u} w \,dx\,dt = \int\!\!\!\!\!\int_{Q} f w \,dx\,dt +\int_0^1 y_0 w(0)  dx, \qquad \forall \, h \in L^2(Q). 
\end{equation}
Hence, \eqref{sol_by_transpo} and \eqref{sol_by_transpo1} imply that $\tilde{y}=\hat{y}$ solves \eqref{sys-nonhom}.

This completes the proof of Theorem \ref{Thm_null_Control_nonhom}.
\end{proof}
We underline that
Theorem \ref{Thm_null_Control_nonhom} provides null controllability property for more regular solution of \eqref{sys-nonhom}. Such a result turns out to be fundamental for the proof of Theorem \ref{Thm_null_Control_memo_1}.

\section{Null controllability for memory system \eqref{sys-memory1}}\label{sect_null_control_memory}
In this section, we analyze the null controllability result for the degenerate parabolic equation \eqref{sys-memory1}. First, for all $k \geq 0$,
we set
\begin{equation}
E_{s,k,R} = \big\{ w \in E_{s,k} : \quad \|(s\beta)^{-k/2} e^{-s\sigma} w \|_{L^2(Q)} \leq R \big\},
\end{equation}
where $R$ is an arbitrary positive constant. 
Clearly, $E_{s,k,R}$ is a bounded, closed and convex subset of $L^2(Q)$.


Let 
$ w \in E_{s,k,R}$ and consider the following  system:
\begin{equation}\label{sys_integ_w}
\left\{
\begin{array}{lll}
\displaystyle y_t - (a(x)  y_x )_x = \int\limits_0^t b(t,s,x)  w(s,x) \, ds + 1_{\omega} u &  & (t, x) \in Q,\\
y(t,1)= 0, & & t \in (0, T),  \\
\begin{cases}
& y(t, 0) = 0,   \qquad   \quad (WD), \\
& (a y_x ) (t, 0)= 0 , \quad (SD), \\
\end{cases}  & & t \in (0, T),\\
y(0,x)= y_{0}(x),  & & x \in  (0, 1).
\end{array}
\right.
\end{equation}

Hence, the next null controllability result holds.
\begin{proposition}\label{Prop_nul_control_sys_w}
Let $T$ and $R$ strictly positive and $k\geq 0$. Assume that the memory kernel  satisfies,  
\begin{equation}\label{hypoth_kernel}
(T-t)^{2k} e^{\left(\frac{4}{T}\right)^4 \frac{s\gamma d}{(T-t)^4}} b \in L^{\infty}((0, T) \times Q),
\end{equation}
where $\gamma$ and $d$ are the constants of \eqref{weightfunc_deg} and $s$ is the same of Lemma \ref{Lemma_modif_Carl_estimate}. Then, for all $w \in E_{s,k R}$ and for any $y_0 \in H_a^1(0,1)$, there exists $u \in L^2(Q)$ such that the associated solution $y$ of system \eqref{sys_integ_w} 
belongs to $ E_{s,k}$.
\end{proposition}

Notice that,  condition \eqref{hypoth_kernel} may appear as a quite strong restriction on the admissible kernel function $b$. Notwithstanding, it is instead a natural one, since the only thing that we are asking is its integrability with respect to the Carleman weight. In other words, $b$  should decay exponentially to $0$ as $t$ goes to $T^-$. Recall that this assumption is less restrictive, for larger values of
the parameter $k > 0$, than \eqref{hyp-kernel-k0}. 

\begin{proof}
Let 
$w \in E_{s,k,R}$ and let $y\in Z_T$ the solution of \eqref{sys_integ_w}.
Using the fact that $- \gamma d \beta \leq \Phi $ in $Q$ (see \eqref{caract_of_psi}), we get that
\begin{align*}
\int\!\!\!\!\!\int_{Q} & (s\beta)^{-k} e^{-2s\Phi} \Big(\int\limits_0^t b(t,s,x)  w(s,x) \, ds \Big)^2 \,dx\,dt \notag\\
&\leq C_T \int\!\!\!\!\!\int_{Q} \int\limits_0^t (s\beta)^{-k} e^{-2s\Phi} b^2(t,s,x) w^2(s,x) \, ds \,dx\,dt \notag\\
&\leq C_T \int\!\!\!\!\!\int_{Q} \int\limits_0^t (s\beta)^{-k} e^{2s \gamma d \beta} b^2(t,s,x) w^2(s,x) \, ds \,dx\,dt \notag\\
&\leq C_T  s^{-k} \int\!\!\!\!\!\int_{Q} \int\limits_0^t (T-t)^{4k} e^{\frac{2s\gamma d}{(T/4)^4(T-t)^4}} b^2(t,s,x) w^2(s,x) \, ds \,dx\,dt.
\end{align*}
Hence, by virtue of condition \eqref{hypoth_kernel}, we have
\begin{align}\label{estim_integral_bw}
\int\!\!\!\!\!\int_{Q} (s\beta)^{-k} e^{-2s\Phi} & \Big(\int\limits_0^t b(t,s,x)  w(s,x) \, ds \Big)^2 \,dx\,dt \notag\\
&\leq C_T s^{-k} \int\!\!\!\!\!\int_{Q} w^2  \,dx\,dt;
\end{align}
therefore, using H\"{o}lder's inequality, the fact that $\sup\limits_{(t,x) \in \overline{Q}}(s \beta(t))^{k}  e^{2s\sigma}(t,x) < +\infty$ and $w \in E_{s,k,R}$, we conclude that
\begin{align*}
\int\!\!\!\!\!\int_{Q} & (s\beta)^{-k} e^{-2s\Phi} \Big(\int\limits_0^t b(t,s,x)  w(s,x) \, ds \Big)^2 \,dx\,dt \notag\\
&\leq C_T s^{-k} \Big(\sup\limits_{(t,x) \in \overline{Q}}(s \beta(t))^{k}  e^{2s\sigma}(t,x) \Big) \int\!\!\!\!\!\int_{Q}  (s \beta)^{-k}  e^{-2s\sigma} w^2 \,dx\,dt \notag \\
& \leq C_T s^{-k} R^2 < +\infty.
\end{align*}

This implies that $ (s\beta)^{-k/2} e^{-s\Phi} \Big(\int\limits_0^t b(t,s,x)  w(s,x) \, ds \Big) \in L^2(Q)$. Hence, in view of Theorem \ref{Thm_null_Control_nonhom}, we deduce that there exists $u \in L^2(Q)$ such that the associated solution $y$ of \eqref{sys_integ_w} belongs to
$ E_{s,k}$.
Hence, the conclusion follows.
\end{proof}
As a consequence of Proposition \ref{Prop_nul_control_sys_w} and Kakutani's fixed point Theorem, we obtain the following result.

\begin{theorem}\label{Thm_null_Control_memo_1}
Let $T >0$, $k\geq 0$ and assume that \eqref{hypoth_kernel} holds with $s\ge s_0$ such that
\[
C s^{-k} e^{-s\gamma \left(\frac{2}{T} \right)^8  d^{\ast}} \leq \frac{1}{2},
\] 
where $\gamma$, $d^*$ and $C$ are the constants that appear in \eqref{condition_d_gamma} and \eqref{estim_solut_control}, respectively.

Then, for any $y_0 \in H_a^1(0,1)$, there exists $u \in L^2(Q)$ such that the associated solution $y\in Z_T$ of \eqref{sys-memory1} satisfies
$$ y(T, \cdot) = 0 \qquad \text{in} \; (0, 1).$$
\end{theorem}
\begin{remark}
\begin{itemize}
\item Let us recall that, without any hypothesis on the kernel $b$, the null controllability of \eqref{sys-memory1} fails (see \cite{Guerrero2013,Zhou2014}). Hence, the decaying condition \eqref{hypoth_kernel} could be necessary.
\item A condition similar to \eqref{hypoth_kernel} already appears in the work of Q. Tao and H. Gao in  \cite{TaoGao16} for uniformly parabolic equations (see \eqref{hyp_kernel_Tao_Gao}). Hence, the null controllability result stated in Theorem \ref{Thm_null_Control_memo_1} for the degenerate equation with memory can be seen as an extension to the one obtained in \cite{TaoGao16}.
\item The difference on the powers of the exponential terms in \eqref{hypoth_kernel} and \eqref{hyp_kernel_Tao_Gao} is mainly due to the different weighted time functions considered in these two contexts.
\item
 Owing to Remark \ref{remrk_on_theta}, we can actually decrease the exponent $4$ in the assumption \eqref{hypoth_kernel} to the exponent $2$. In particular, in place of  \eqref{hypoth_kernel} we can assume
\begin{equation*}
(T-t)^{2k} e^{\frac{\hat{C}}{(T-t)^2}} b \in L^{\infty}((0, T) \times Q),
\end{equation*}
where $\hat{C}= \big(\frac{4}{T}\big)^2 s\gamma d$.
\end{itemize}
\end{remark}

\begin{proof}[Proof of Theorem \ref{Thm_null_Control_memo_1}]
For the moment take 
$R>0$ sufficiently large. Define, as in \cite{TaoGao16}, the multivalued mapping $ \Lambda : E_{s,k,R} \subset E_{s,k} \rightarrow 2^{E_{s,k}} $ in the following
way: for every $w \in E_{s,k,R}$, $\Lambda(w)$ is the set of $y \in E_{s,k}$ such that for some
$u\in L^2(Q)$ satisfying 
\begin{equation}\label{bound_u}
\int\!\!\!\!\!\int_{Q_{\omega}} (s\beta)^{-(k+3)} e^{-2s\sigma} u^2 \,dx dt \leq Ce^{2s [\widehat{\Phi}(0)-\Phi^*(\frac{5T}{8})]} \Big(R^2+ s^{-k} \int_0^1 e^{-2s\widehat{\Phi}(0)} y_0^2\,dx   \Big),
\end{equation}
 the associated solution $y$ of \eqref{sys_integ_w} satisfies 
\begin{equation}\label{null_state_y}
y \in E_{s,k} \quad\text{and} \quad y(T, \cdot) = 0 \qquad \text{in} \; (0, 1).
 \end{equation}

Thus, our task is reduced to prove that $\Lambda $ admit at least one fixed point in $E_{s,k,R}$. To this aim, it suffices to check that $\Lambda$ satisfies the assumptions of Kakutani's fixed point Theorem. Next, we are going to check that all the conditions to apply such a theorem
in $L^2(Q)$ topology are satisfied.

Clearly, $\Lambda(w)$ is a closed set of $L^2(Q)$. Moreover, thanks to  Proposition \ref{Prop_nul_control_sys_w}, $\Lambda(w)$ is non empty. The fact that the identity in \eqref{null_state_y} is stable by convex combinations yields the convexity of $ \Lambda(w)$.

Now, let us prove that $\Lambda(E_{s,k,R}) \subset E_{s,k,R}$ for a sufficiently large $R$. Using the inequality \eqref{estim_solut_control}, condition \eqref{hypoth_kernel} and proceeding as in \eqref{estim_integral_bw}, we have
\begin{align*}
\int\!\!\!\!\!\int_{Q} & (s \beta)^{-k}  e^{-2s\sigma} y^2  \,dx\,dt +  \int\!\!\!\!\!\int_{Q_{\omega}} (s\beta)^{-(k+3)} e^{-2s\sigma} u^2 \,dx dt \notag \\
&\qquad  \leq Ce^{2s [\widehat{\Phi}(0)-\Phi^*(\frac{5T}{8})]} \Big(\int\!\!\!\!\!\int_{Q} (s\beta)^{-k} e^{-2s\Phi} \Big(\int\limits_0^t b(t,s,x)  w(s,x) \, ds \Big)^2 \,dx\,dt \notag\\
& \qquad \qquad \qquad \qquad \qquad + s^{-k} \int_0^1 e^{-2s\widehat{\Phi}(0)} y_0^2\,dx   \Big)\notag \\
&\qquad  \leq Ce^{2s [\widehat{\Phi}(0)-\Phi^*(\frac{5T}{8})]} \Big(s^{-k} \int\!\!\!\!\!\int_{Q}   w^2 \,dx\,dt + s^{-k} \int_0^1 e^{-2s\widehat{\Phi}(0)} y_0^2\,dx   \Big).
\end{align*}
Therefore, applying H\"{o}lder's inequality, we obtain
\begin{align*}
\int\!\!\!\!\!\int_{Q} & (s \beta)^{-k}  e^{-2s\sigma} y^2  \,dx\,dt +  \int\!\!\!\!\!\int_{Q_{\omega}} (s\beta)^{-(k+3)} e^{-2s\sigma} u^2 \,dx dt \notag \\
&  \leq C s^{-k} e^{2s [\widehat{\Phi}(0)-\Phi^*(\frac{5T}{8})]} \bigg(\Big(\sup\limits_{(t,x) \in \overline{Q}}(s \beta(t))^{k}  e^{2s\sigma(t,x)} \Big)
\Big(\int\!\!\!\!\!\int_{Q}  (s \beta)^{-k}  e^{-2s\sigma} w^2 \,dx\,dt \Big)\\
&\qquad \qquad\qquad \qquad \qquad + \int_0^1 e^{-2s\widehat{\Phi}(0)} y_0^2\,dx   \bigg).
\end{align*}
In particular, since $ \Phi \leq \widehat{\Phi}(0)$, $ \sigma \leq \widehat{\sigma}(0) \; \text{in} \, Q$ (see \eqref{modif_weight_funct}), 
$\sup\limits_{(t,x) \in \overline{Q}}(s \beta(t))^{k}  e^{\frac{s}{2}\sigma(t,x)} < + \infty
$
and $w \in E_{s,k,R}$, the last inequality becomes
\begin{align}\label{inequality_u_y}
\int\!\!\!\!\!\int_{Q} & (s \beta)^{-k}  e^{-2s\sigma} y^2  \,dx\,dt +  \int\!\!\!\!\!\int_{Q_{\omega}} (s\beta)^{-(k+3)} e^{-2s\sigma} u^2 \,dx dt \notag \\
&\qquad  \leq C s^{-k} \bigg(e^{s[2 \widehat{\Phi}(0)-2 \Phi^*(\frac{5T}{8})+\frac{3}{2} \widehat{\sigma}(0)]} R^2 + e^{-2 s \Phi^*(\frac{5T}{8})} \int_0^1  y_0^2\,dx   \bigg).
\end{align}
On the other hand, by taking the parameter $\rho$ in \eqref{weightfunc_nondeg} so that 
$\rho > \frac{ln3}{\|\sigma\|_{\infty}},$
one can show that the interval $\displaystyle\Big(\frac{e^{2 \rho \|\sigma\|_{\infty}}}{d-d^\star}, \, \frac{3( e^{2 \rho \|\sigma\|_{\infty}} - e^{ \rho \|\sigma\|_{\infty}})}{2 (d-d^\star)}
\Big)$ is nonempty, and thus, we can choose the constant $\gamma$ ( see \eqref{weightfunc_deg}) in such a way
$$
\frac{e^{2 \rho \|\sigma\|_{\infty}} -1}{d-d^\ast} < \gamma < \frac{3\big( e^{2 \rho \|\sigma\|_{\infty}} - e^{\rho \|\sigma\|_{\infty}} \big)}{2(d-d^\star)} .
$$
Thus, as a straightforward consequence, one has
\begin{equation}\label{compar_max_min}
\frac{3}{2}\hat{\sigma}(t) \leq \hat{\Phi}(t) \qquad \text{for every}\quad t\in (0,T).
\end{equation}
Using \eqref{compar_max_min}, the definitions of $ \widehat{\Phi}$ and $\Phi^*$, and choosing $d \geq 10 d^{\ast}$, we find
\begin{align*}
2 \widehat{\Phi}(0)-2 \Phi^*(\frac{5T}{8})+\frac{3}{2} \widehat{\sigma}(0) & \leq  3 \widehat{\Phi}(0)-2 \Phi^*(\frac{5T}{8}) \\
& = 3 \gamma (d^{\ast} -d) \beta(0) + 2 \gamma d \beta\Big(\frac{5T}{8}\Big)\\
& = \gamma \Big(\frac{2}{T} \Big)^8 \Big[ 3 (d^{\ast} -d) ) + 2  d \Big(\frac{16}{15} \Big)^4 \Big]\\
& = \gamma \Big(\frac{2}{T} \Big)^8 \Big[3 d^{\ast}  - d \Big(3 - 2 \Big(\frac{16}{15} \Big)^4\Big) \Big]\\
& < -  \gamma \Big(\frac{2}{T} \Big)^8  d^{\ast} < 0. 
\end{align*}
By assumption $s$ is such that
$$ C s^{-k} e^{s[2 \widehat{\Phi}(0)-2 \Phi^*(\frac{5T}{8})+\frac{3}{2} \widehat{\sigma}(0)]} \leq \frac{1}{2},$$
thus, we immediately obtain, from this last inequality and \eqref{inequality_u_y}, 
\begin{align}\label{estim_u_y_of_sys_w}
\int\!\!\!\!\!\int_{Q}  (s \beta)^{-k}  e^{-2s\sigma} y^2 & \,dx\,dt +  \int\!\!\!\!\!\int_{Q_{\omega}} (s\beta)^{-(k+3)} e^{-2s\sigma} u^2 \,dx dt \notag \\
&\qquad \leq \bigg(\frac{1}{2} R^2+ C e^{-2 s \Phi^*(\frac{5T}{8})} s^{-k} \int_0^1 y_0^2\,dx \bigg).
\end{align}
Hence, for $R$  sufficiently large, we have
\begin{align*}
\int\!\!\!\!\!\int_{Q} & (s \beta)^{-k}  e^{-2s\sigma} y^2  \,dx\,dt  \leq R^2.
\end{align*}
As a consequence, $\Lambda(E_{s,k,R}) \subset E_{s,k,R}$.

Furthermore, let $\{w_n\}$ be a sequence of $E_{s,k,R}$. Thanks to Proposition \ref{prop-Well-posed_nonhom}, the associated solutions $\{y_n\}$ are bounded in  $Z_T$. Then, in view of Aubin-Lions Theorem, this implies that $\Lambda(E_{s,k,R})$ is relatively compact in $L^2(Q)$.

Let us finally check that $\Lambda$ is upper-semicontinuous under the $L^2$ topology. To this aim, let $\{w_n\}$ be a sequence satisfying $w_n \rightarrow w$ in $ E_{s,k,R}$ and $ y_n \in \Lambda(w_n)$ such that $y_n \rightarrow y$ in $L^2(Q)$. Our objective is to prove that $ y \in \Lambda(w)$.
At first, observe that for any $w_n \in E_{s,k,R}$, we can find at least one control $u_n \in L^2(Q)$ such that the associated solution $y_n$ belongs to $L^2(Q)$. By virtue of  Proposition \ref{prop-Well-posed_nonhom} and \eqref{estim_u_y_of_sys_w}, we deduce that there is a subsequence satisfying
\begin{align}
 u_n \rightarrow u \quad &\text{weakly in} \; L^2(Q)\notag \\
 y_n \rightarrow \tilde{y} \quad & \text{weakly in} \; Z_T \quad \text{and} \\
 &\text{strongly in} \; C(0, T; L^2(0,1)).
\end{align}
This yields $y= \tilde{y} $ in $L^2(Q)$. 

Since $(y_n,u_n)$ satisfies the system 
\begin{equation}\label{sys_integ_wn}
\left\{
\begin{array}{lll}
\displaystyle y_{n,t} - (a(x)  y_{n,x} )_x = \int\limits_0^t b(t,s,x)  w_n(s,x) \, ds + 1_{\omega} u_n, &  & (t, x) \in Q, \\
y_n(t,1)= 0, & & t \in (0, T),  \\
\begin{cases}
& y_n(t, 0) = 0,   \qquad   \quad (WD), \\
& (a y_{n,x} )   (t, 0)= 0 , \quad (SD), \\
\end{cases}  & & t \in (0, T),\\
y_n(0,x)= y_{0}(x),  & & x \in  (0, 1),
\end{array}
\right.
\end{equation}
hence passing to weak limit, it follows that the couple $(y,u)$ satisfies \eqref{sys_integ_w}. This provides that $y \in \Lambda(w)$ and, therefore, $\Lambda$ is upper semicontinuous.

Consequently, using the Kakutani's fixed point Theorem in the $L^2(Q)$ topology for the mapping $\Lambda$, we infer that there is at least one $ y \in E_{s,k,R}$ such that $y \in \Lambda(y)$. Thus, by the definition of $\Lambda$,  there exists at least one couple $(y,u)$ satisfying all the conditions in Theorem \ref{Thm_null_Control_memo_1}.  The uniqueness of $y$ follows by Proposition \ref{prop-Well-posed_memory}. Hence, the proof of Theorem \ref{Thm_null_Control_memo_1} is complete.
\end{proof}
Clearly, Theorem  \ref{Thm_null_Control_memo_1}  holds also in a general domain $(t^*,T) \times (0,1)$ with suitable changes. Thanks to this fact, the following null controllability result holds for memory system \eqref{sys-memory1}. 
\begin{theorem}\label{Thm_null_Control_memo_2}
Let $T >0$, $k \geq 0$ and assume that \eqref{hypoth_kernel} holds with $s$ as in Theorem \ref{Thm_null_Control_memo_1}.
Then, for any $y_0 \in L^2(0,1)$,
there exists $u \in L^2(Q)$ such that the associated solution $y \in W_T$  of \eqref{sys-memory1} satisfies
$$ \quad y(T, \cdot) = 0 \qquad \text{in} \; (0, 1). $$
\end{theorem}
\begin{proof}
Consider the following homogeneous parabolic problem:
\begin{equation*}
\left\{
\begin{array}{lll}
\displaystyle w_t - (a(x)  w_x )_x = \int\limits_0^t b(t,s,x)  w(s,x) \, ds  &  & (t, x) \in (0, \frac{T}{2}) \times (0, 1), \\
w(t,1)= 0, & & t \in (0, \frac{T}{2}),  \\
\begin{cases}
& w(t, 0) = 0,   \qquad   \quad (WD), \\
& (a w_x )  (t, 0)= 0 , \quad (SD), \\
\end{cases}  & & t \in (0, \frac{T}{2}),\\
w(0,x)= y_{0}(x),  & & x \in  (0, 1),
\end{array}
\right.
\end{equation*}
where $y_0$ is the initial condition in $\eqref{sys-memory1}.$

By Proposition \ref{prop-Well-posed_memory}, the solution of this system belongs to 
$$ W_T^*: = L^2\left(0, \frac{T}{2}; H_a^1(0,1)\right)\cap C\left(\left[0, \frac{T}{2}\right]; L^2(0,1)\right).$$
Then, there exists $t^* \in (0, \frac{T}{2})$ such that
$w(t^*, \cdot) := w^*(\cdot) \in H_a^1(0,1)$.

Now, we consider the following controlled parabolic system:
\begin{equation*}
\left\{
\begin{array}{lll}
\displaystyle z_t - (a(x)  z_x )_x = \int\limits_0^t b(t,s,x)  z(s,x) \, ds + 1_{\omega} h &  & (t, x) \in (t^*, T) \times (0, 1), \\
z(t,1)= 0, & & t \in (t^*, T),  \\
\begin{cases}
& z(t, 0) = 0,   \qquad   \quad (WD), \\
& (a z_x )  (t, 0)= 0 , \quad (SD), \\
\end{cases}  & & t \in (t^*, T),\\
z(t^*,x)= w^*(x),  & & x \in  (0, 1).
\end{array}
\right.
\end{equation*}
Hence, thanks to Theorem \ref{Thm_null_Control_memo_1}, there exists  $h \in L^2((t^*, T) \times (0, 1))$ such that the associated solution $z \in Z_T^* : =  L^2(t^*, T; H_{a}^{2}(0,1))\cap H^1(t^*, T; L^2(0,1))$ satisfies
$$ \quad z(T, \cdot) = 0 \qquad \text{in} \; (0, 1). $$
Finally, setting 
\begin{align*}
y:=
\left\{
\begin{array}{lll}
w, & \text{in} \quad \big[0, t^*\big],\\
z, & \text{in} \quad \big[t^*, T\big]
\end{array}
\right. 
 \quad  \text{and} \quad u:=
\left\{
\begin{array}{lll}
0, & \text{in} \quad \big[0, t^*\big],\\
h, & \text{in} \quad \big[t^*, T\big],
\end{array}
\right. 
\end{align*}
one can prove that $y\in W_T$ solves the system \eqref{sys-memory1} associated to $u$ and is such that
$$ \quad y(T, \cdot) = 0 \qquad \text{in} \; (0, 1). $$
Hence, the thesis follows.
\end{proof}

\begin{remark}
In the present context, by Theorem \ref{Thm_null_Control_memo_2}, one can deduce  immediately the null controllability result for \eqref{sys-memory1} when the control acts at the {\it nondegenerate}
point $x = 1$. Indeed, it is sufficient to use a standard technique and a localization argument as in \cite[Remark 4.6.2]{Alabau2006}. Of course, the situation is completely different in the case when the control acts at the {\it degenerate} point $x=0$. We refer to \cite{CMV2016} for a discussion of this issue. 
\end{remark}

\begin{remark}
Observe that, as in the context of parabolic equation without memory (i.e., $b=0$), the null controllability for \eqref{sys-memory1} proved in Theorem \ref{Thm_null_Control_memo_2}
yields the exact controllability to trajectories, that is, for any trajectory $\overline{y}$ (i.e. solution of \eqref{sys-memory1} corresponding to $u\equiv0$ and $y_0 \in L^2(0,1)$) and any $y_0 \in L^2(0,1)$, there exists $u\in L^2(Q)$ such that the associated   solution to \eqref{sys-memory1} satisfies 
$$ \overline{y}(T, x) = y(T, x), \qquad x\in (0, 1).$$

Indeed, let us consider a trajectory $\overline{y}$ and introduce the following change of variables $z=y - \overline{y}$, where $y$ is a solution of \eqref{sys-memory1}.
Hence, $z$ satisfies the following controlled system:
\begin{equation*}
\left\{
\begin{array}{lll}
\displaystyle z_t - (a(x)  z_x )_x = \int\limits_0^t b(t,s,x)  z(s,x) \, ds + 1_{\omega} u &  & (t, x) \in Q, \\
z(t,1)= 0, & & t \in (0, T),  \\
\begin{cases}
& z(t, 0) = 0,   \qquad   \quad (WD),\\
& (a z_x )  (t, 0)= 0 , \quad (SD), \\
\end{cases}  & & t \in (0, T),\\
z(0,x)= z_0(x),  & & x \in  (0, 1),
\end{array}
\right.
\end{equation*}
where $z_0=y_0 - \overline{y}_0$.

According to Theorem \ref{Thm_null_Control_memo_2} there exists $u \in L^2(Q)$ such that 
$$ z(T, x) = 0, \qquad x\in (0, 1). $$
Consequently,
$$ \overline{y}(T, x) = y(T, x), \qquad x\in (0, 1). $$
\end{remark}
\section{Comments}\label{sect_comments}
In this section we discuss some extensions of the above null controllability results and describe some perspectives that are related to this paper.
\subsection{Null controllability in the case $a(1)=0$}
In this subsection we address the null controllability result for the following degenerate parabolic equation with memory
\begin{equation}\label{sys-memory2}
\left\{
\begin{array}{lll}
\displaystyle y_t - (a(x)  y_x )_x = \int\limits_0^t b(t,s,x)  y(s,x) \, ds + 1_{\omega} u &  & (t, x) \in Q, \\
y(t,0)= 0, & & t \in (0, T),  \\
\begin{cases}
& y(t, 1) = 0,   \qquad   \quad (WD), \\
& (a y_x ) (t, 1)= 0 , \quad (SD), \\
\end{cases}  & & t \in (0, T),\\
y(0,x)= y_{0}(x),  & & x \in  (0, 1),
\end{array}
\right.
\end{equation}
where $y_0 \in L^2(0, 1)$ and $a$ degenerates at the extremity $x=1$, i.e.,
$a(1)=0$. 
In order to present our main result we need to introduce the functional spaces where our problem will be well posed. As before, we distinguish the two following cases:
\begin{itemize}
\item  Weakly degenerate case (WD)
\begin{equation}\label{hyp_WD_right}
\left\{
\begin{array}{lll}
a \in C([0,1]) \cap C^1([0,1)), \; a(1)=0, \;  a>0 \quad \text{in}\quad [0, 1), \\
\exists \; \tilde{\alpha} \in [0,1),\quad \text{such that}\quad (x-1) a'(x) \leq \tilde{\alpha} a(x),\quad \forall \; x \in [0,1],
\end{array}
\right.
\end{equation}
\item  Strongly degenerate (SD)
\begin{equation}\label{hyp_SD_right}
\left\{
\begin{array}{lll}
a \in C^1([0,1]), \; a(1)=0, \; a>0 \quad \text{in}\quad [0, 1), \\
\exists \, \tilde{\alpha} \in [1,2),\quad \text{such that}\quad (x-1) a'(x) \leq \tilde{\alpha} a(x),\quad \forall \, x \in [0,1],\\
\left\{
\begin{array}{ll}
\exists \, \tilde{\beta} \in (1, \tilde{\alpha}],\, x \mapsto \dfrac{a(x)}{(1 - x)^{\tilde{\beta}}}\quad \text{is nonincreasing near}\quad 0,\quad \text{if}\quad \tilde{\alpha} > 1, \\
\exists \, \tilde{\beta} \in (0, 1),\, x \mapsto \dfrac{a(x)}{(1 - x)^{\tilde{\beta}}}\quad \text{is nonincreasing near}\quad 0,\quad \text{if}\quad \tilde{\alpha} = 1.
\end{array}
\right.
\end{array}
\right.
\end{equation}
\end{itemize}
Clearly, the prototype is $ a(x) = (1 - x)^{\tilde{\alpha}}, \quad \tilde{\alpha} \in (0, 2). $

Let us introduce the weighted spaces $H_a^1$ and $H_a^2$ as follows:

Case (WD).
\begin{align*}
H_a^1:= \Big\{ y \in L^2(0,1): y \; \text{a.c. in} \; [0,1], \quad\sqrt{a}y_x \in L^2(0,1) \;\text{and}\, y(0)=y(1)=0 \Big\}
\end{align*}
and
\begin{align*}
H_a^2:= \Big\{ y \in H_a^1(0, 1): ay_x \in H^1(0,1)\Big\}.
\end{align*}
Case (SD).
\begin{align*}
H_a^1:= \Big\{ y \in L^2(0,1): y\; \text{ locally a.c. in}\;[0,1), \quad\sqrt{a}y_x \in L^2(0,1)\;\text{and}\; y(0)=0 \Big\}
\end{align*}
and
\begin{align*}
H_a^2:&= \Big\{ y \in H_a^1(0, 1): ay_x \in H^1(0,1)\Big\}\\
&=\Big\{ y \in L^2(0,1): y \; \text{locally a.c. in} \; [0,1), \quad ay\in H^1_0(0,1),\\
&\qquad ay_x \in H^1(0,1) \; \text{and} \; (ay_x)(1)=0 \Big\}.
\end{align*}

Using the above spaces, one can prove that the  well-posedness results given in Propositions \ref{prop-Well-posed_nonhom} and \ref{prop-Well-posed_memory} still  hold. On the contrary, setting $\varphi := \theta \tilde{\psi}$, where $\theta$ is defined as in \eqref{weightfunc_deg} and 
\begin{equation}\label{funzioni}
\tilde{\psi} := \tilde{\gamma} \Big(\int_{x}^{1}\frac{1- y}{a(y)}\,dy - \tilde{d} \Big),
\end{equation}
with $\tilde{\gamma}$ and $\displaystyle \tilde{d} > \int_{0}^{1}\frac{1- y}{a(y)}\,dy$ positive constants,  the next  null controllability result holds.
\begin{theorem}\label{Thm_null_Control_memo_3}
Let $T >0$, $k \geq 0$ and assume that
\begin{equation}\label{hypoth_kernel1}
(T-t)^{2k} e^{\big(\frac{4}{T}\big)^4\frac{ s\tilde \gamma \tilde d}{(T-t)^4}} b \in L^{\infty}((0, T) \times Q),
\end{equation}
with $s$ as in Theorem \ref{Thm_null_Control_memo_1}.
 Then, for any $y_0 \in L^2(0,1)$,
there exists $u \in L^2(Q)$ such that the associated solution $y \in W_T$ of \eqref{sys-memory2} satisfies
$$ \quad y(T, \cdot) = 0 \qquad \text{in} \; (0, 1). $$
\end{theorem}
\begin{proof}
The proof of this theorem follows the same strategy of Theorem \ref{Thm_null_Control_memo_2}; of course using symmetric arguments. The main difference is that here, in place of  \eqref{harpoideginequality} and \eqref{Carl_estimate-1}, we use the  following Hardy Poincar\'e inequality:

{\it there is a positive constant $C$ such that, for every $y\in H_a^1(0,1)$, the following inequality holds
\begin{equation*}
\int_0^1\frac{a(x)}{(1 - x)^2}y^2(x)\,dx\leq C\int_0^1a(x)|y_x(x)|^2\,dx,
\end{equation*}
}
and the following Carleman estimate:

{\it there exist two positive constants $C$ and $s_0$, such that the solution $v \in Z_T$ of \eqref{sys-adj-nonhom} satisfies
\begin{align*}
\int\!\!\!\!\!\int_{Q} \Big(s \theta  a(x) v_{x}^{2} & + s^{3} \theta^3 \frac{(1 -x)^2}{ a(x)} v^{2}\Big) e^{2s\varphi}\,dx\,dt \notag \\ &\leq C\Big(\int\!\!\!\!\!\int_{Q} g^2 e^{2s\varphi}\,dx\,dt  +  \int\!\!\!\!\!\int_{Q_{\omega}} s^2 \theta^2 v^2 e^{2s\varphi} \,dx dt\Big)
\end{align*}
for all $s \ge s_0$.}

As the procedure is completely similar, we omit the details of the proof.
\end{proof}

\subsection{Null controllability in the case  $a(0)=a(1)=0$}
In this subsection we will extend the null controllability result proved above to the degenerate parabolic equation with memory
\begin{equation}\label{sys-memory3}
\left\{
\begin{array}{lll}
\displaystyle y_t - (a(x)  y_x )_x = \int\limits_0^t b(t,s,x)  y(s,x) \, ds + 1_{\omega} u &  & (t, x) \in Q, \\
\begin{cases}
& y(t, 0) = 0 = y(t, 1),   \qquad   \qquad (WWD), \\
& (a y_x ) (t, 0) = 0 = y(t, 1) , \qquad (SWD), \\
& y(t, 0) = 0 = (a y_x ) (t, 1) , \qquad (WSD), \\
& (a y_x ) (t, 0) = 0 = (a y_x ) (t, 1) ,  \quad (SSD), \\
\end{cases}  & & t \in (0, T),\\
y(0,x)= y_{0}(x),  & & x \in  (0, 1),
\end{array}
\right.
\end{equation}
where $y_0 \in L^2(0, 1)$ and $a$ vanishes at both extremities of the interval $(0, 1)$ and satisfies, as in \cite{MV2006}, one of the four following cases:
\begin{itemize}
\item weakly-weakly degenerate case (WWD): 
\begin{equation*}
\left\{
\begin{array}{lll}
a \in C([0,1]) \cap C^1((0,1)), \; a(0)=a(1)=0, \; a>0 \quad \text{in}\quad (0, 1), \\
\exists \, \alpha \in [0,1),\quad \text{such that}\quad x a'(x) \leq \alpha a(x),\quad \forall \, x \in [0,1],\\
\exists \; \tilde{\alpha} \in [0,1),\quad \text{such that}\quad (x-1) a'(x) \leq \tilde{\alpha} a(x),\quad \forall \; x \in [0,1],
\end{array}
\right.
\end{equation*}
\item strongly-weakly degenerate case (SWD):
\begin{equation*}
\left\{
\begin{array}{lll}
a \in C([0,1]) \cap C^1([0,1)), \; a(0)=a(1)=0, \; a>0 \quad \text{in}\quad (0, 1), \\
\exists \, \alpha \in [1,2),\quad \text{such that}\quad x a'(x) \leq \alpha a(x),\quad \forall \, x \in [0,1],\\
\left\{
\begin{array}{ll}
\exists \, \beta \in (1, \alpha],\, x \mapsto \dfrac{a(x)}{x^{\beta}}\quad \text{is nondecreasing near}\quad 0,\quad \text{if}\quad \alpha > 1, \\
\exists \, \beta \in (0, 1),\, x \mapsto \dfrac{a(x)}{x^{\beta}}\quad \text{is nondecreasing near}\quad 0,\quad \text{if}\quad \alpha=1,\\
\end{array}
\right.\\
\exists \; \tilde{\alpha} \in [0,1),\quad \text{such that}\quad (x-1) a'(x) \leq \tilde{\alpha} a(x),\quad \forall \; x \in [0,1],
\end{array}
\right.
\end{equation*}
\item weakly-strongly degenerate case (WSD): 
\begin{equation*}
\left\{
\begin{array}{lll}
a \in C([0,1]) \cap C^1((0,1]), \; a(0)=a(1)=0, \; a>0 \quad \text{in}\quad (0, 1), \\
\exists \, \alpha \in [0,1),\quad \text{such that}\quad x a'(x) \leq \alpha a(x),\quad \forall \, x \in [0,1],\\
\exists \, \tilde{\alpha} \in [1,2),\quad \text{such that}\quad (x-1) a'(x) \leq \tilde{\alpha} a(x),\quad \forall \, x \in [0,1],\\
\left\{
\begin{array}{ll}
\exists \, \tilde{\beta} \in (1, \tilde{\alpha}],\, x \mapsto \dfrac{a(x)}{(1 - x)^{\tilde{\beta}}}\quad \text{is nonincreasing near}\quad 0,\quad \text{if}\quad \tilde{\alpha} > 1, \\
\exists \, \tilde{\beta} \in (0, 1),\, x \mapsto \dfrac{a(x)}{(1 - x)^{\tilde{\beta}}}\quad \text{is nonincreasing near}\quad 0,\quad \text{if}\quad \tilde{\alpha} = 1.
\end{array}
\right.
\end{array}
\right.
\end{equation*}
\item strongly-strongly degenerate case (SSD): 
\begin{equation*}
\left\{
\begin{array}{lll}
a \in C^1([0,1]), \; a(0)=a(1)=0, \; a>0 \quad \text{in}\quad (0, 1), \\
\exists \, \alpha \in [1,2),\quad \text{such that}\quad x a'(x) \leq \alpha a(x),\quad \forall \, x \in [0,1],\\
\left\{
\begin{array}{ll}
\exists \, \beta \in (1, \alpha],\, x \mapsto \dfrac{a(x)}{x^{\beta}}\quad \text{is nondecreasing near}\quad 0,\quad \text{if}\quad \alpha > 1, \\
\exists \, \beta \in (0, 1),\, x \mapsto \dfrac{a(x)}{x^{\beta}}\quad \text{is nondecreasing near}\quad 0,\quad \text{if}\quad \alpha=1,\\
\end{array}
\right.\\
\exists \, \tilde{\alpha} \in [1,2),\quad \text{such that}\quad (x-1) a'(x) \leq \tilde{\alpha} a(x),\quad \forall \, x \in [0,1],\\
\left\{
\begin{array}{ll}
\exists \, \tilde{\beta} \in (1, \tilde{\alpha}],\, x \mapsto \dfrac{a(x)}{(1 - x)^{\tilde{\beta}}}\quad \text{is nonincreasing near}\quad 0,\quad \text{if}\quad \tilde{\alpha} > 1, \\
\exists \, \tilde{\beta} \in (0, 1),\, x \mapsto \dfrac{a(x)}{(1 - x)^{\tilde{\beta}}}\quad \text{is nonincreasing near}\quad 0,\quad \text{if}\quad \tilde{\alpha} = 1.
\end{array}
\right.
\end{array}
\right.
\end{equation*}
\end{itemize}
A typical example is $ a(x) = x^{\alpha} (1 - x)^{\tilde{\alpha}}, \quad \text{with} \; \alpha, \tilde{\alpha} \in [0,2)$.

As previously, in order to study the well-posedness of problem \eqref{sys-memory3}, we shall define four different classes of weighted spaces.

Case (WWD).
\begin{align*}
H_a^1:= \Big\{ y \in L^2(0,1): y\; \text{a.c. in}\, [0,1],
\quad\sqrt{a}y_x \in L^2(0,1) \; \text{and} \; y(0)=y(1)=0 \Big\}
\end{align*}
and
\begin{align*}
H_a^2:= \Big\{ y \in H_a^1(0, 1): ay_x \in H^1(0,1)\Big\}.
\end{align*}
Case (SWD).
\begin{align*}
H_a^1:= \Big\{ y \in L^2(0,1): y\; \text{locally a.c. in}\;(0,1], \quad\sqrt{a}y_x \in L^2(0,1)\;\text{and}\; y(1)=0 \Big\}
\end{align*}
and
\begin{align*}
H_a^2:&= \Big\{ y \in H_a^1(0, 1): ay_x \in H^1(0,1)\Big\}\\
&=\Big\{ y \in L^2(0,1): y \; \text{locally a.c. in} \; (0,1], \quad ay\in H^1_0(0,1),\\
&\qquad ay_x \in H^1(0,1) \; \text{and} \; (ay_x)(0)=0 \Big\}.
\end{align*}
Case (WSD).
\begin{align*}
H_a^1:= \Big\{ y \in L^2(0,1): y\; \text{locally a.c. in}\;[0,1), \quad\sqrt{a}y_x \in L^2(0,1)\;\text{and}\; y(0)=0 \Big\}
\end{align*}
and
\begin{align*}
H_a^2:&= \Big\{ y \in H_a^1(0, 1): ay_x \in H^1(0,1)\Big\}\\
&=\Big\{ y \in L^2(0,1): y \; \text{locally a.c. in} \; [0,1), \quad ay\in H^1_0(0,1),\\
&\qquad ay_x \in H^1(0,1) \; \text{and} \; (ay_x)(1)=0 \Big\}.
\end{align*}
Case (SSD).
\begin{align*}
H_a^1:= \Big\{ y \in L^2(0,1): y \; \text{locally a.c. in}\;(0,1), \quad\sqrt{a}y_x \in L^2(0,1) \Big\}
\end{align*}
and
\begin{align*}
H_a^2:&= \Big\{ y \in H_a^1(0, 1): ay_x \in H^1(0,1)\Big\}\\
&=\Big\{ y \in L^2(0,1): y \; \text{locally a.c. in} \; (0,1), \quad ay\in H^1_0(0,1),\\
&\qquad ay_x \in H^1(0,1) \; \text{and} \; (ay_x)(0)= (ay_x)(1)=0 \Big\}.
\end{align*}
Again, the well-posedness results proved in Propositions \ref{prop-Well-posed_nonhom} and \ref{prop-Well-posed_memory} still hold and, as a consequence of Theorems \ref{Thm_null_Control_memo_2} and \ref{Thm_null_Control_memo_3}, one can deduce the following null controllability result for \eqref{sys-memory3}.
\begin{theorem}\label{Thm_null_Control_memo_4}
 Let $T >0$, $k \geq 0$ and assume 
\begin{equation}\label{hypoth_kernel2}
(T-t)^{2k} e^{\big(\frac{4}{T}\big)^4 \frac{ s\bar\gamma \bar d}{(T-t)^4}} b \in L^{\infty}((0, T) \times Q),
\end{equation}
with $s$ as in Theorem \ref{Thm_null_Control_memo_1}, where  $\bar \gamma = \max\{\gamma, \tilde \gamma\}, \bar d = \max \{d, \tilde d\}$. Then, for any $y_0 \in L^2(0,1)$,
there exists $u \in L^2(Q)$ such that the associated solution $y \in W_T$ of \eqref{sys-memory3} satisfies
$$ \quad y(T, \cdot) = 0 \qquad \text{in} \; (0, 1). $$
Here $\gamma, \tilde \gamma, d$ and $\tilde d$ are the constants given in \eqref{weightfunc_deg} and in \eqref{funzioni}.
\end{theorem}
 \begin{proof}
Consider the following parabolic system
 \begin{equation}\label{P1}
\left\{
\begin{array}{lll}
\displaystyle w_t - (a(x)  w_x )_x = \int\limits_0^t b(t,s,x)  w(s,x) \, ds + 1_{\omega} u_1 &  & (t, x) \in (0, T) \times (0, \beta'), \\
w(t,\beta')= 0, & & t \in (0, T),  \\
\begin{cases}
& w(t, 0) = 0,   \qquad   \quad (WD), \\
& (a w_x )  (t, 0)= 0 , \quad (SD), \\
\end{cases}  & & t \in (0, T),\\
w(0,x)= y_{0}(x),  & & x \in (0, \beta'),
\end{array}
\right.
\end{equation}
where $\omega \Subset (\lambda', \beta') \Subset (0, 1)$ and $y_0$ is the initial condition in \eqref{sys-memory3}.

Thus, by Theorem \ref{Thm_null_Control_memo_2}, we know that there exists a control $u_1 \in L^2((0, T) \times (0, \beta')) $ such that the associated solution $w \in W_T$ of \eqref{P1} satisfies
$$ w(T, \cdot) =0, \quad \text{in} \, (0, \beta').$$
Now, define $\tilde{w}$ the trivial extension of $w$ in $[0, 1]$. Hence
$$ \tilde{w}(T, \cdot) =0, \quad \text{in} \, (0, 1).$$
In a similar way, we consider the following parabolic system
 \begin{equation}\label{P2}
\left\{
\begin{array}{lll}
\displaystyle z_t - (a(x)  z_x )_x = \int\limits_0^t b(t,s,x)  z(s,x) \, ds + 1_{\omega} u_2 &  & (t, x) \in (0, T) \times (\lambda' , 1), \\
z(t,\lambda')= 0, & & t \in (0, T),  \\
\begin{cases}
& z(t, 1) = 0,   \qquad   \quad (WD), \\
& (a z_x )  (t, 1)= 0 , \quad (SD), \\
\end{cases}  & & t \in (0, T),\\
z(0,x)= y_{0}(x),  & & x \in (\lambda' , 1).
\end{array}
\right.
\end{equation}
Then, thanks to Theorem \ref{Thm_null_Control_memo_3}, there exists a control $u_2 \in L^2((0, T) \times (\lambda' , 1)) $ such that the associated solution $z \in W_T$ solution of \eqref{P2} satisfies
$$ z(T, \cdot) = 0, \quad \text{in} \, (\lambda' , 1).$$
Now, define $\tilde{z}$ the trivial extension of $z$ in $[0, 1]$. Hence
$$ \tilde{z}(T, \cdot) = 0, \quad \text{in} \, (0, 1).$$
Next, consider
\begin{equation*}
\tilde{u}_1(t,x)=  \left\{
\begin{array}{ll}
\displaystyle
u_1 (t,x), & (t,x) \in (0, T) \times (0, \beta'), \\
0, & (t,x) \in (0, T) \times (\beta', 1),
\end{array}
\right.
\end{equation*}
\text{and} 
\begin{equation*}
\tilde{u}_2 (t,x)=  \left\{
\begin{array}{ll}
\displaystyle
0, & (t,x) \in (0, T) \times (0, \lambda'), \\
u_2 (t,x), & (t,x) \in (0, T) \times (\lambda', 1).
\end{array}
\right.
\end{equation*}
Let $\chi \in C^\infty([0, 1])$ be a smooth cut-off function
 such that
\begin{equation}\label{xi}
0\leq \chi(x) \leq 1 , \quad
\chi(x)=  \left\{
\begin{array}{ll}
\displaystyle
1, & x \in (0, \lambda''), \\
0, & x \in (\beta'', 1),
\end{array}
\right.
\end{equation}
where $(\lambda'', \beta'') \Subset \omega$
and set $ y = \chi \tilde{w} + (1 - \chi) \tilde{z}$.

Then, one can easily verifies that
\begin{equation*}
y_t = \chi \tilde{w}_t + (1 - \chi) \tilde{z}_t,
\end{equation*}
and
\begin{align*}
 (a y_x)_x = & \chi (a \tilde{w}_x )_x + (1 - \chi) (a \tilde{z}_x )_x +
((a \tilde{w})_x \chi_x + a \tilde{w} \chi_{xx} +  a \tilde{w}_x \chi_x )  \\
& - ((a \tilde{z})_x \chi_x + a \tilde{z} \chi_{xx} +  a \tilde{z}_x \chi_x ).
\end{align*}
Therefore, we find that
\begin{align*}
y_t - (a y_x)_x - & \int\limits_0^t b(t,s,x)  y(s,x) \, ds = \chi \Big(\tilde{w}_t - (a \tilde{w}_x )_x - \int\limits_0^t b(t,s,x)  \tilde{w}(s,x) \, ds \Big) \\
& + (1 - \chi) \Big(\tilde{z}_t - (a \tilde{z}_x )_x - \int\limits_0^t b(t,s,x) \tilde{z}(s,x) \, ds \Big) \\ 
& - \Big((a \tilde{w})_x \chi_x + a \tilde{w} \chi_{xx} +  a \tilde{w}_x \chi_x \Big)
+ \Big((a \tilde{z})_x \chi_x + a \tilde{z} \chi_{xx} +  a \tilde{z}_x \chi_x \Big)\\
& = 1_{\omega} \chi u_1 + 1_{\omega} (1 - \chi) u_2 - \Big((a \tilde{w})_x \chi_x + a \tilde{w} \chi_{xx} +  a \tilde{w}_x \chi_x \Big) \\
& \quad + \Big((a \tilde{z})_x \chi_x + a \tilde{z} \chi_{xx} +  a \tilde{z}_x \chi_x \Big).
\end{align*}
Observe that the supports of $\chi_x$ and $\chi_{xx}$ are contained in $(\lambda'', \beta'') \Subset \omega$. Then, we can write
\begin{align*}
y_t - (a y_x)_x = \int\limits_0^t b(t,s,x)  y(s,x) \, ds + 1_{\omega} u
\end{align*}
where $ u \in L^2(Q)$ satisfies
\begin{align*}
1_{\omega} u & = 1_{\omega} \chi \tilde{u}_1 + 1_{\omega} (1 - \chi) \tilde{u}_2 - \Big((a \tilde{w})_x \chi_x + a \tilde{w} \chi_{xx} +  a \tilde{w}_x \chi_x \Big) \\
& \quad + \Big((a \tilde{z})_x \chi_x + a \tilde{z} \chi_{xx} +  a \tilde{z}_x \chi_x \Big).
\end{align*}
Moreover, using the definitions of $\tilde{w}$, $\tilde{z}$ and $\chi$, it follows that
\begin{align*}
& y(t, 0) = \Big(\chi \tilde{w} + (1 - \chi) \tilde{z} \Big)(t,0) = 0, \qquad t\in (0, T), \\
& y(t, 1) = \Big(\chi \tilde{w} + (1 - \chi) \tilde{z} \Big)(t,1) = 0, \qquad t\in (0, T),\\
& (a y)_x(t, 0) = \Big( \chi_x a \tilde{w} + \chi (a \tilde{w}_x) - \chi_x a \tilde{z} + (1 - \chi) (a \tilde{z}_x)\Big) (t,0) = 0,\qquad t\in (0, T),\\
& (a y)_x(t, 1) =\Big( \chi_x a \tilde{w} + \chi (a \tilde{w}_x) - \chi_x a \tilde{z} + (1 - \chi) (a \tilde{z}_x)\Big) (t,1) = 0, \qquad t\in (0, T),
\end{align*}
from which we get the boundary conditions given in \eqref{sys-memory3}.

In addition, we have
\begin{align*}
y(0, x) & = \chi(x) \tilde{w}(0, x) + (1 - \chi(x)) \tilde{z}(0, x) \\
& = \chi(x) y_0(x) + (1 - \chi(x))y_0(x) = y_0(x), \qquad x\in (0, 1).
\end{align*}
In conclusion, $y$ solves the memory system \eqref{sys-memory3}, and satisfies
$$ y(T, \cdot) = \chi \tilde{w}(T, \cdot) + (1 - \chi) \tilde{z}(T, \cdot) = 0 \qquad \text{in} \, (0, 1). $$
Hence the claim follows.
\end{proof}
{\bf Final comments.}
In the context of parabolic equation without memory, i.e., $b=0$, it is well known that, once there exists a control function acting on a control region $\omega \subset (0, 1)$ that drives the system from an initial state $y_0$ to the equilibrium at time $t=T$, i.e., $y(T, \cdot) = 0$, we can stop controlling, by setting $u\equiv 0$ for $t\geq T$, and the underlying system naturally stays at rest for all $t\geq T$, i.e.,
$$ y(t, \cdot) = 0, \qquad  \forall \, t\geq T. $$
Unfortunately, this is not the case for the parabolic equation with memory. Indeed, due to the effect of the accumulated memory at time $t=T$, i.e., $\displaystyle\int\limits_0^T b(T,s,\cdot) y(s, \cdot) \, ds$, the null state of this system at $T$ cannot be kept for $t \geq T$ in the absence of control function.

Hence, it could be of interest to consider a more general concept of null controllability for system of type \eqref{sys-memory1}. In particular, we look for a control function that drives both the state and the memory term to $0$ at time $t=T$. 

This problem has been addressed by S. Ivanov and L. Pandolfi  in \cite{Ivanov2009} for the parabolic equation with memory and through a distributed control:
$$ y_t -  y_{xx} = \int_0^t b(t-s) y_{xx}(s)  ds + 1_{\omega} u, \qquad (t,x) \in Q. $$
In \cite{Ivanov2009} it is proved that, this system cannot be controlled
to rest for large classes of memory kernels and controls. In fact, the presence of the memory terms makes the controllability of this system to be impossible if the control is located in a fixed subset $\omega$. 

On the other hand, to obtain controllability result as explained in  \cite{Chaves2017} and \cite{Chaves2014}, the support of the control function needs to move to cover the domain where the equation evolves in the control time horizon. We refer to \cite{Chaves2017} where this problem is discussed in the context of heat equation. The extension to the degenerate problem is the subject of a future work. 


\begin{thebibliography}{99}
\bibitem{Hajjaj2013} 
E. M. Ait Ben Hassi, F. Ammar Khodja, A. Hajjaj and L. Maniar, {\it Carleman estimates and null controllability of coupled degenerate systems}, Evol. Equ. Control Theory, 2 (2013), 441–459.
\bibitem{Alabau2006} 
F. Alabau-Boussouira, P. Cannarsa, G. Fragnelli, {\it Carleman estimates for degenerate parabolic operators with application to null controllability},  J. Evol. Equ. 6 (2006), 161-204.
\bibitem{AHMS19} 
B. Allal, A. Hajjaj, L. Maniar, J. Salhi,
{\it Lipschitz stability for some coupled degenerate parabolic systems with locally distributed observations of one component}, Math. Control \& Rela. Fields, doi: 10.3934/mcrf.2020014.
\bibitem{Barbu2000} V. Barbu, M. Iannelli, {\it Controllability of the heat equation with memory}, Differential Integral Equations 13 (2000) 1393–1412.
\bibitem{Campiti1998} M. Campiti, G. Metafune and D. Pallara,
{\it Degenerate self-adjoint evolution equations on the unit interval}, Semigroup Forum, 57 (1998), 1-36.
\bibitem{bfm2018}
I. Boutaayamou, G. Fragnelli, L. Maniar, {\it  Carleman estimates for parabolic equations with interior degeneracy and Neumann boundary conditions},  J. Anal. Math., \textbf{135} (2018), 1--35.
\bibitem{CMV2008}
P. Cannarsa, P. Martinez and J. Vancostenoble,
{\it Carleman estimates for a class of degenerate parabolic operators},
SIAM J. Control Optim. 47 (2008), 1-19.
\bibitem{Carmelo2000} J. Carmelo Flores, Luz De Teresa, {\it Null controllability of one dimensional degenerate parabolic equations with first order terms}, Discrete \& Continuous Dynamical Systems-B, (2020)
\bibitem{CMV2016}
P. Cannarsa, P. Martinez and J. Vancostenoble, {\it Global Carleman estimates for degenerate parabolic operators with applications}, Mem. Amer. Math. Soc. 239 (2016), ix+209 pp.
\bibitem{Chaves2017} F. Chaves-Silva, X. Zhang and E. Zuazua, {\it Controllability of evolution equations with memory, SIAM Journal on Control and Optimization}, 55(2017), 2437–2459, doi: 10.1137/151004239.
\bibitem{Chaves2014} F. W. Chaves-Silva, L. Rosier and E. Zuazua. {\it Null controllability of a system of viscoelasticity with a moving control}, J. Math. Pures Appl. 101 (2014), 198-222.
\bibitem{Fadili} M. Fadili and L. Maniar, {\it Null controllability of $n$-coupled degenerate parabolic systems with $m$-controls}, J. Evol. Equ., 17 (2017), 1311-1340.
\bibitem{fm2013}
G. Fragnelli, D. Mugnai,  Carleman estimates and observability
inequalities for parabolic equations with interior degeneracy,
{\it Advances in Nonlinear Analysis }\textbf{2} (2013), 339--378.
\bibitem{fm2016}
G. Fragnelli, D. Mugnai, Carleman estimates, observability inequalities
and null controllability for interior degenerate
non smooth parabolic equations,
{\it Mem. Amer. Math. Soc.} \textbf{242} (2016), v+84 pp. {\it Corrigendum}, to appear.
\bibitem{FI1996}
 A. V. Fursikov and O. Y. Imanuvilov, {\it
 Controllability of evolution equations},
 \emph{Lect. Notes Ser. 34, Seoul National University,
Seoul, 1996.}
\bibitem{Grasselli1991} 
M. Grasselli and A. Lorenzi, {\it Abstract nonlinear Volterra integro-differential equations with nonsmooth kernels}, Atti. Accad. Naz. Lincei Cl. Sci. Fis. Mat. Natur. Rend. Lincei (9) Mat. Appl., 2 (1991), 43–53
\bibitem{Guerrero2013}
S. Guerrero and O. Yu. Imanuvilov,  {\it Remarks on non controllability of the heat equation with memory}, ESAIM Control Optim. Calc. Var. 19 (2013), 288-300.
\bibitem{Ivanov2009} 
S. Ivanov and L. Pandolfi,  {\it Heat equation with memory: Lack of controllability to rest}. J. Math. Anal. Appl. 355 (2009) 1–11.
\bibitem{Lavanya2009} 
R. Lavanya, K. Balachandran, {\it Null controllability of nonlinear heat equations with memory effects}, Nonlinear Anal. Hybrid Syst. 3 (2009) 163–175.
\bibitem{Lions71}
J.L. Lions,  {\it Optimal control of systems governed by partial differential equations}, Springer-Verlag, 1971.
\bibitem{Lions83}
J.L. Lions,  {\it Contr\^ole des syst\`emes distribu\'es singuliers}, Gauthier-Villars,
Paris, 1983.
\bibitem{Zuazua2017} Q. L\"{u}, X. Zhang, E. Zuazua, {\it Null controllability for wave equations with memory}, J. Math. Pures Appl. (9) 108 (2017), no. 4, 500-531.


\bibitem{MV2006}
P. Martinez, J. Vancostenoble,  {\it Carleman estimates for one-dimensional
degenerate heat equations}, J. Evol. Eq. 6 (2006), 325–362.
\bibitem{Munoz2003} J.E. Mu\~{n}oz Rivera, M.G. Naso, {\it Exact boundary controllability in thermoelasticity with memory}, Adv. Difference Equ. 8 (2003) 471–490.
\bibitem{Saktiv2008}
K. Sakthivel, K. Balachandran, B.R. Nagaraj, {\it On a class of non-linear parabolic control systems with memory effects},
Internat. J. Control 81 (2008) 764-777.
\bibitem{TaoGao16}
Q. Tao and H. Gao, {\it On the null controllability of heat equation with memory}, J. Math. Anal. Appl. 440 (2016) 1-13.
\bibitem{Yong2005} J. Yong and X. Zhang, {\it Exact controllability of the heat equation with hyperbolic memory kernel}, in: Control Theory of Partial Differential Equations, in: Lect. Notes Pure Appl. Math., vol. 242, Chapman $\&$ Hall/CRC, Boca Raton, FL, 2005, pp. 387–401.
\bibitem{Zhou2014} 
X. Zhou and H. Gao, {\it Interior approximate and null controllability of the heat equation with memory}, Comput. Math. Appl. 67 (2014), 602-613.
\bibitem{Zhou2018} 
X. Zhou and M. Zhang, {\it on the controllability of a class of degenerate parabolic equations with memory}, J Dyn Control Syst 24, 577–591 (2018).  https://doi.org/10.1007/s10883-017-9382-7.


\end{thebibliography}
\end{document}